\documentclass[12pt,a4paper,oneside,reqno]{amsart}

\usepackage[margin=2cm]{geometry}
\usepackage{float}
\usepackage{hyperref}
\usepackage{amsmath,amsfonts}
\usepackage{amsthm}
\usepackage{amssymb}
\usepackage{url}

\def\Z{\mathbb{Z}}

\newcommand{\dsum}{\displaystyle\sum}

\newtheorem{theorem}{Theorem}

\usepackage[ruled,vlined]{algorithm2e}
\usepackage{tikz}

\usepackage{pgfplots}
\pgfplotsset{compat=newest}
\usetikzlibrary{decorations.markings}

\usepackage{setspace}

\usetikzlibrary{shapes,snakes}

\usepackage{multicol}
\usepackage{multirow}

\let\origmaketitle\maketitle
\def\maketitle{
  \begingroup
  \def\uppercasenonmath##1{} 
  \let\MakeUppercase\relax 
  \origmaketitle
  \endgroup
}

\begin{document}

\title[Reallocating and Sharing Health Equipments in Sanitary Emergency Situations]{\Large Reallocating and Sharing Health Equipments in Sanitary Emergency Situations:\\ The COVID-19 Case in Spain}

\author[V. Blanco, R. G\'azquez, \MakeLowercase{and} M. Leal]{{\large V\'ictor Blanco$^\dagger$, Ricardo G\'azquez$^\dagger$ \MakeLowercase{and} Marina Leal$^\ddagger$}\medskip\\
$^\dagger$Institute of Mathematics, Universidad de Granada\\
$^\ddagger$Dpt. Statistics, Mathematics and Informatics, Universidad Miguel Hernández}

\address{Institute of Mathematics, Universidad de Granada}
\email{vblanco@ugr.es, rgazquez@ugr.es}

\address{Dpt. Statistics, Mathematics and Informatics, Universidad Miguel Hernández}
\email{m.leal@umh.es}

\date{\today}

\begin{abstract}
In this paper we provide a mathematical programming based decision tool to optimally reallocate and share equipments between different units in emergency situations under lack of resources. The approach is motivated by the COVID-19 pandemic in which many Heath National Systems were not able to satisfy the demand of ventilators, sanitary individual protection equipments or different human resources. Our tool is based in two main principles: (1) Part of the stock of equipments at a unit that is not needed (in near future) could be shared to other units; and (2) extra stock to be shared among the units in a region can be efficiently distributed taking into account the demand of the units. The decisions are taken with the aim of minimizing certain measures of the non-covered demand in a region where a given network structured set of units is given. The mathematical programming models that we provide are stochastic and multiperiod and we provide different robust objective functions.  Since the proposed models are computationally hard to solve, we provide a \textit{divide-et-conquer} math-heuristic approach.  We report the results of applying our approach to the data of the COVID-19 case in different regions of Spain, highlighting some interesting conclusions of our analysis, such as the great increase of treated patients if the proposed redistribution tool is applied.

\end{abstract}


\maketitle

\section{Introduction}

The high contagiousness of the COVID-19 virus has caused the rapid propagation of the virus all over the world; and this has led to the sudden need of hospitalization of a large amount of patients affected by COVID-19, many of them in Intensive Care Units (ICU). This abrupt and massive increase in the number of hospitalizations has bring some hospital to the collapse and has produced lack of material and human resources, such as diagnostic tests, ventilators, ICU beds, etc., in many hospitals. This extreme situation has given rise to the application of improvised measures, such as the creation of field hospitals or ageisming when assigning (insufficient) ventilators to patients \cite{italia2020,espana2020}, even though this second measure is not allowed by the World Health Organization \cite{who}. Other improvised measure, in Spain for instance, were the hire of ventilators between hospitals in different cities just to cover the demand at a specific time of the pandemic, without taking into account demand forecasting \cite{22andalucia,10extremadura,9murcia,11galicia}, or distributing test packages between regions by population instead of by number of infected citizens \cite{global}, etc. In these situations in which the number of COVID-19 infected citizens and the gravity of their disease are distributed heterogeneously in a country or state, it is necessary to provide efficient strategies to distribute equipment and share the available resources between different hospitals or units in order to efficiently satisfy the demands.

Allocating/reallocating resources is a recurrent application of Operations Research and a wide variety of situations and tools analyzing the efficient  distribution of goods can be found in literature  (see, e.g.  \cite{bodson02,gomar02,hegazy99,michaud08,wang11}, among many others), although most of them are cost-oriented, i.e., they are designed to minimize some function of the transportation or set-up costs. However, in emergency situations, satisfying in adequate time the needs of the society affected by any damage, beyond its distribution cost, is crucial. Thus, an efficient resource allocation and collaboration among the different units which may supply materials is essential. In the COVID-19 situation, the life of sick patients depends on the use of ventilators, and the non-infection of health personnel depends on the availability of sanitary protection equipments. Therefore, designing strategies to share equipments between hospitals will help to reduce the impact of this crisis and avoid a lack of human resources able to cut the pandemic. Furthermore, national and regional governments are receiving extra stock of equipments, provided  by international suppliers or by different private initiatives (as the home-made confection of face-masks,  3D-printing of helmets, etc.) and sharing this stock among the different units should be done by mean of the demand of each of the units instead of population-based sharing, as those that have been applied.  In September 2020, the pandemic has already provoked close to one million worldwide deaths~\cite{wiki}, which reflects the urgent necessity of cooperation of the different agents to palliate the effect of the virus.

In this paper we propose a general framework for the reallocation of any type of equipments and sharing available and extra stock between the different units. Our proposal allows us to generalize some of the approaches that have been recently proposed, including them as special cases. In particular, recently in \cite{mehrotra2020model}, the authors propose a stochastic multiperiod planning model to allocate and reallocate ventilators to treat critical patients. They perform it for a particular network with a central agency taking the reallocation decision and federal agencies and states deciding on the percentage of available ventilators to share, by minimizing the expected non-covered demand. 

The model that we provide considers four main ingredients:
\begin{description}
\item[Distribution Network:] The units and the different types of available logistic platforms are linked in a network-based input information, allowing us to model different policies stablished for the regions in a country. We assume that a network where the products can be distributed is given. 
\item[Multiperiod:] The planning is performed for a short time horizon, being the approach a multiperiod decision-making tool.
\item[Uncertainty:] Since the demand of equipments in each unit is stochastic along time (it depends on the number of sick or infected citizen), we incorporate this uncertainty in the model.
\item[Robust Decisions:] Our model decides on the optimal way to reallocate and share extra stock based on different aggregation measures of the non-covered demand along the planing horizon, minimizing different robust objective functions, as the the maximum (by units) non-covered expected demand; the maximum (by time periods) non-covered expected demand, the maximum (by regions) non-covered expected demand; the overall non-covered expected demand\footnote{Note that although overall non-covered expected demand is not formally a robust function, for the sake of generality we will also analyze this case.}, and also their minimax regret counterparts that provide robust solutions with a \textit{good} performance under any of the possible uncertain situations that may happen.
\end{description} 

Apart from the above, our model has different particularities that can be adequately tuned to be adapted to each situation. We limit the number of deliveries from each unit to  avoid the use of excessive resources needed to load, unload and mount equipments. We also upper bound the amount of equipments to be delivered from each unit to a percentage of the available stock at each period, allowing hospitals to fix such a percentage based on their risk level.

It is clear that Mathematical Programming plays an important role in the design of optimal strategies to (re)distribute goods and determine sharing policies.The above characteristics of our model would require to incorporate to our methodology different elements which makes the problem specially difficult to solve. On the one hand, our problem will be modeled as a Stochastic programming problem~\cite{birge2011introduction}, in which a function of the expected non-covered demand (which is the uncertain data in our problem) is minimized. Second, we want to find fair solutions that do not harm the weakest units, using Robust objective functions, of type minmax~\cite{campbell2008routing,ogryczak2000inequality,ye2017fair} and minmax regret~\cite{con07, con18, gut96, lopez13}. Also, our problem consists of a planning model in multiple periods, being necessary to deal with  multiperiod problems~\cite{BENMOHAMED2020102,GHOLAMI2020102297,PUN2019754,SHIN2019102170}. 


The rest of the paper is organized as follows. In Section \ref{sec:2}, we describe the input elements to derive a model for the problem. Section \ref{sec:3} is devoted to present the mathematical programming models for the problem. The math-heuristic approach is described in Section \ref{sec:4}. The analysis of our model for the Spanish data is reported in Section \ref{sec:5}. Finally, we draw some conclusions and further extensions in Section \ref{sec:6}.

\section{Preliminaries}\label{sec:2}

We analyze here the problem of distributing goods on a network with several particularities, as described above. In this section we describe the elements involved in the problems and introduce the notation and the problem under study. 

Let us consider a weighted directed graph $G=(N,A; \mathbf{w})$ where:
\begin{itemize}
\item $N=\{1, \ldots, n\}$ represents the different units in the distribution system. Some nodes may represent hospitals or (local, regional or national) health logistic centers. 
\item $A$, the set of arcs, indicates the available direct links between units. 
\item $\mathbf{w}$, the set of weights, origin-destination times (in days) to transfer equipments between the arcs. These weights include the times needed to load, deliver, unload and mount the equipments.
\end{itemize}
For the sake of presentation we assume that a single type of product is distributed along the network, although our approach can be easily adapted to distribute different types of goods (in which case, a set of weights for each product being distributed must be provided).

In Figure \ref{fig:1}, we illustrate a hierarchical network that may represent the situation in many countries when distributing health equipments. There, solid circles represent different hospitals. They are supplied by local logistic centers (stars) which are at the same time supplied by regional logistic centers (squares). All the regions are supplied by a national logistic center (empty circle). The lines indicate the links in which the products can be distributed. 

\begin{figure}
    \begin{center}
\begin{tikzpicture}[scale=0.4]
    \node[star,fill=black, inner sep=1.5pt,minimum size=1pt] at (0,-1) (center1) {};
    \node[star,fill=black, inner sep=1.5pt,minimum size=1pt] at (-1,3) (center2) {};
    \node[star,,fill=black, inner sep=1.5pt,minimum size=1pt] at (3,-1) (center3) {};
    \node[diamond,draw, inner sep=1.5pt,minimum size=1pt] at (0.5,1) (center) {};
    
    \draw (center)--(center1);
    \draw (center)--(center2);
    \draw (center)--(center3);
    
    \foreach \n in {1,...,5}{
        \node[circle,fill=black, inner sep=1.5pt,minimum size=1pt] at ([shift=({\n*360/5 + 25}:1cm)]center1) (n\n)  {};
        \draw (center1)--(n\n);
    }
    
    \foreach \n in {1,...,7}{
        \node[circle,fill=black, inner sep=1.5pt,minimum size=1pt] at ([shift=({\n*360/7 + 10}:1cm)]center2) (n\n)  {};
        \draw (center2)--(n\n);
    }
        \foreach \n in {1,...,3}{
        \node[circle,fill=black, inner sep=1.5pt,minimum size=1pt] at ([shift=({\n*360/3}:1cm)]center3) (n\n)  {};
        \draw (center3)--(n\n);
    }
    
        \node[star,fill=black, inner sep=1.5pt,minimum size=1pt] at (7,-1) (center1a) {};
    \node[star,fill=black, inner sep=1.5pt,minimum size=1pt] at (8,3) (center2a) {};
    \node[star,,fill=black, inner sep=1.5pt,minimum size=1pt] at (10,0) (center3a) {};
    \node[diamond,draw, inner sep=1.5pt,minimum size=1pt] at (7.5,1) (centera) {};
    
    \draw (centera)--(center1a);
    \draw (centera)--(center2a);
    \draw (centera)--(center3a);
    
    \foreach \n in {1,...,5}{
        \node[circle,fill=black, inner sep=1.5pt,minimum size=1pt] at ([shift=({\n*360/5 + 25}:1cm)]center1a) (n\n)  {};
        \draw (center1a)--(n\n);
    }
    
    \foreach \n in {1,...,7}{
        \node[circle,fill=black, inner sep=1.5pt,minimum size=1pt] at ([shift=({\n*360/7 + 10}:1cm)]center2a) (n\n)  {};
        \draw (center2a)--(n\n);
    }
        \foreach \n in {1,...,3}{
        \node[circle,fill=black, inner sep=1.5pt,minimum size=1pt] at ([shift=({\n*360/3}:1cm)]center3a) (n\n)  {};
        \draw (center3a)--(n\n);
    }
    
        \node[star,fill=black, inner sep=1.5pt,minimum size=1pt] at (0,6) (center1b) {};
    \node[star,fill=black, inner sep=1.5pt,minimum size=1pt] at (2,10) (center2b) {};
    \node[star,,fill=black, inner sep=1.5pt,minimum size=1pt] at (5,8) (center3b) {};
    \node[diamond,draw, inner sep=1.5pt,minimum size=1pt] at (2.5,8) (centerb) {};
    
    \draw  (centerb)--(center1b);
    \draw (centerb)--(center2b);
    \draw (centerb)--(center3b);
    
    \foreach \n in {1,...,5}{
        \node[circle,fill=black, inner sep=1.5pt,minimum size=1pt] at ([shift=({\n*360/5 + 25}:1cm)]center1b) (n\n)  {};
        \draw  (center1b)--(n\n);
    }
    
    \foreach \n in {1,...,7}{
        \node[circle,fill=black, inner sep=1.5pt,minimum size=1pt] at ([shift=({\n*360/7 + 10}:1cm)]center2b) (n\n)  {};
        \draw  (center2b)--(n\n);
    }
        \foreach \n in {1,...,3}{
        \node[circle,fill=black, inner sep=1.5pt,minimum size=1pt] at ([shift=({\n*360/3}:1cm)]center3b) (n\n)  {};
        \draw (center3b)--(n\n);
    }
    
    \draw (centera)--(centerb);
    \draw (centera)--(center);
    \draw (center)--(centerb);
    
    \node[circle,draw, inner sep=1.5pt,minimum size=1pt] at (3,4) (centerX) {};
    
    \draw (centerX)--(center);
    \draw (centerX)--(centera);
    \draw (centerX)--(centerb);
    
    \end{tikzpicture}
    \end{center}
\caption{Example of the graph structure involving different types of units and links in a health distribution system.\label{fig:1}}
\end{figure}
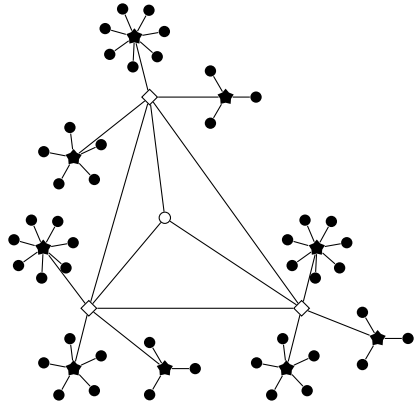

From the graph $G$, we consider the set of pairs of nodes for which there is a path in the graph between them, i.e.
$$
W=\{(i,j): i,j\in N \text{ such that there exist a path from $i$ to $j$}\}.
$$ 

In order to provide an optimal sharing of extra stocks and to account for the different policies of the different regions, we assume that the set of nodes, $N$,  is divided into $p$ groups, $N_1, \ldots, N_p \subseteq N$, not necessarily disjoint, which may represent different regions or clusters of hospitals which may receive extra stock from the same source. Observe that one unit may belong to different regions and receive equipments from different sources. This would be the case of countries that centrally share equipments directly to all (or some) hospitals but also regions with extra or available stock share equipments to the hospitals in the region. In these case, a hospital in the region belongs to two of the $N$-sets, to the one in which all the hospitals belong to and also to the one in which only the hospitals of the regions are included.

In the example drawn in Figure \ref{fig:1}, the different circles linked to a same star may represent one of these sets (provinces), all which are connected to a star linked to the same square represent others sets (regions), and all the square nodes linked with empty circle node is the last set (country). We denote by $P=\{1, \ldots, p\}$, the index set for the groups of regions.

We also consider the following list of input parameters:

\begin{itemize}
\item $T = \{1, \ldots, q\}$: planning horizon (in days). We will perform a distribution planning for a finite number, $q$, of periods.
\item $s_i^0$: initial stock at node $i$,  for all $i\in N$. Each unit is assumed to have this nonnegative initial stock at the beginning of the planning horizon and that is available to be used or delivered. It represents the number of equipment at each unit to cover the demand.
\item $q_k^t$: amount of extra stock to be shared in period $t$ between the nodes in group $N_k$, for $k \in P$. We assume that the units on group $N_j$ may receive a nonnegative number of units at each period. It allows to model situations in which certain regions buy or receive equipment at certain periods and they want to share them among the units that belong to the region. 
\item $Q_i$: upper bound on the number of deliveries from node $i$ to different units, for  $i\in N$. In order to avoid a large (and unrealistic) number of deliveries from units with low demand but high initial stocks, we will restrict the deliveries from each unit. This upper bound will be induced by the capacity of each unit to prepare and load the equipment to deliver.
\item $\ell_{ij}$: length of the shortest path in the graph with origin $i$ and destination $j$, for all $(i,j)\in W$. This length represents the time (in days) needed to deliver an equipment from unit $i$ to unit $j$ on the graph $G$. 
\item $\gamma_i$: proportion of the available stock (after covering its demand) that unit $i$ is willing to delivered to other units, for $i\in N$.
\item $g_i$: upper bound on the number of equipments to be deliver from unit $i$, for $i\in N$.  
\item $a_i$: storage capacity of unit $i$, for $i\in N$.
\end{itemize}

Note that the parameters $Q_i$, $\gamma_i$,  $g_i$ and $a_i$ do not depend on the time period. However, the time dependence could be easily incorporated in the model, if needed.

\subsection*{Demands}

Apart from the above deterministic information, we consider that the demands of equipments required by the units at each period are, as expected, uncertain. We denote the demands as $d(\xi)$, with $\xi$ a random variable. We assume that $\xi$ has finite support, i.e., that a finite number of possible realizations for the demands is possible. Let us denote by $\Omega(T)$ the finite support of possible scenarios, and by $d_i^{t\omega}$ the demand for node $i\in N$ at period $t\in T$ under scenario $\omega\in \Omega(T)$. Let us also denote by $p^\omega$ the probability of scenario $\omega\in \Omega(T)$, with $p^{\omega}\geq 0$, and $\sum_{\omega \in \Omega(T)} p^{ \omega}=1$. Abusing of notation, unless necessary, we will obviate the index set $T$ in the definition of the set of scenarios, i.e., we will denote by $\Omega \equiv \Omega(T)$.

Although we assume in our model that each scenario contains the information of the demands for all nodes and all periods, different assumptions could be done, being our approach still valid.  For instance, it can also be assumed that there exist a different set of possible scenarios for the demands for each period and each region, or a set of possible scenarios for each period (independently of node or the region). For example, in this last case, for each $t\in T$ there would exist a set of scenarios $\Omega_t$. And for each $\omega \in \Omega_t$ we would have a vector of demands $d^{t\omega }=(d_1^{t\omega},\ldots, d_n^{t\omega})$, with an associate probability  $p^{t\omega}$ with $\dsum_{\omega \in \Omega^t}p^{t\omega} =1$.

With the above input information, our approach consists of determining an \textit{optimal} redistribution planning for the equipments between units and a sharing policy by minimizing a function that takes into account the demand that is not able to be covered at each unit and period by the lack of resources. 

\section{Mathematical Programming Model}\label{sec:3}

In this section we provide a family of mathematical programming models to make decisions on: (1) the amount of product to be delivered between the different nodes of the network at each of the periods of the time horizon; and (2) the sharing of extra amounts received at each time period between the different groups $N_1, \ldots, N_p$, by minimizing different measures of the non-covered demand along the time horizon. Due to the existence of uncertainty in the demands, we propose robust solutions that perform \textit{well} under any scenario.  We describe the different mathematical programming approaches which differ on the measures of the non-covered demand that are considered but share the same sets of variables and constraints.
\subsection{Variables}

We consider the following decision variables in our models.

\begin{itemize}
\item $x_{ij}^{t}$: amount of equipments to deliver from unit $i$ to unit $j$ at period $t$,  for all $(i,j)\in W$, $t\in T$. 
\item $s_{ik}^{t}$: stock of equipments received by node $i$ from the sharing of group $N_k$ at time period $t$, for all $i\in N_k$, $k\in \{1, \ldots, p\}$, $t \in T$.
\end{itemize}

Observe that the above decision variables do not depend on the scenarios, since our aim is to provide a solution (distribution) with a good behaviour under any of the scenarios that may occur. 

We also consider the following auxiliary variables, which can be derived using the above decision variables, to ease the exposition of the models:

\begin{itemize}
\item $y_{ij}^{t } = \left\{\begin{array}{cl} 1 &\mbox{if at least one equipment is delivered from $i$ to $j$ at period $t$},\\
0 & \mbox{otherwise,}
\end{array}\right.$, for all {$(i,j)\in W$}, $t\in T$. 

This variable allows us to control the different loads of equipments from a given unit. In particular $\dsum_{j\in N: (i,j) \in W} y_{ij}$ is the overall number of loads that are prepared at period $t$ from unit $i$, and will be upper bounded to avoid an excess of loads from each unit. Observe that this variables can be obtained from the $x$-variables as:
$$
y_{ij}^{t } = \left\{\begin{array}{cl} 1 &\mbox{if $x_{ij}^t\geq 1$},\\
0 & \mbox{otherwise}
\end{array}\right., \quad \text{ for all } (i,j)\in W, t\in T.
$$

\item $S_i^{t}$: accumulated stock in node $i$ until period $t$:
$$
S_i^{t}=s_i^0 + \sum_{k \in P:\atop i \in N_k}\sum_{t^\prime \leq t} s_{ik}^{t^\prime }, \quad \text{for all } i \in N, t\in T.
$$
It is the initial stock plus the extra shared stock received by each of the groups where $i$ belongs to.
\item $R_i^{t}$: amount of product received until period $t$ by unit $i$ from other units:
$$
R_i^{t } = {\dsum_{j\in N:\atop (j,i)\in W}} \dsum_{t^\prime \leq t:\atop t^\prime  + \ell_{ji}\leq t} x_{ji}^{t^\prime } , \quad \text{for all } i \in N, t\in T.
$$
It is the overall sum on all the units from what $i$ is able to receive equipments and in all the periods $t^\prime$ in which the equipments are delivered plus the delivering time are previous or equal to $t$.
\item $D_{i}^{t}$: amount of product delivered until period $t$ by unit $i$:
$$
D_{i}^{t} =  {\dsum_{j\in N:\atop (i,j)\in W}}\dsum_{t^\prime\leq t} x_{ij}^{t^\prime}, \quad \text{for all } i \in N, t\in T.
$$
\item $H_i^{t \omega}$: effective excess at time period $t$ from unit $i$ under scenario $\omega \in \Omega$:
$$
H_i^{t \omega} = \max\{ 0, S_i^{t } + R_i^{t } - D_i^{(t-1) } -d_i^{t \omega} \}, \quad \text{for all } \omega \in \Omega, i \in N, t\in T,$$ 
where $D_i^0 =0$. That is, the stock received until this period, plus the amount received until this period, minus the delivered until the previous period, minus the demand in this period. If the demand is not covered, $S_i^{t } + R_i^{t } - D_i^{(t-1)} -d_i^{t \omega}<0$, then the excess is $0$. This amount represents the number of equipments that are available at the units in each period (under every scenario) after covering the demand at that unit.
\item Non-Covered Demand of unit $i$ at time period $t$ under scenario $\omega \in \Omega$:
$$
NCD_i^{t \omega} = d_i^{t \omega} + D_i^{t } -S_i^{t } - R_i^{t }, \quad \text{for all } \omega \in \Omega, i \in N, t\in T.
$$
Which is computed as the demand of the unit at that period plus the amount of equipments delivered from the unit at that period minus the cumulated stock and the product received until that period. In case $NCD_i^{t \omega} >0$,  the demand plus the delivered are greater than the amount received, being not desirable and producing lack of resources at that period. Otherwise, if $NCD_i^{t \omega}<0$, the demand plus the number of deliveries is less than the received amount, then the demand $d_{i}^{t\omega}$, can be covered with the available equipments.
\item Nonnegative Non-Covered Demand of demand point $i$ at time period $t$ under scenario $\omega \in \Omega$:
$$
\overline{NCD}_i^{t \omega} =\max\{0, {NCD}_i^{t \omega}\}, \quad \text{for all } \omega \in \Omega, i \in N, t\in T.
$$
As mentioned above, the actual demand that is not covered at a unit at a given time period under a given scenario is represented only when $NCD_i^{t \omega}>0$. Thus, this auxiliary variable consider only that positive part, in case it exists, and zero otherwise. These variables will be used in our objective functions instead of $NCD_i^{t \omega}$ to be somehow minimized, since the negative non-covered demand (which represents the positive stock) is not convenient to be minimized because it may provoke an excess of stock in the units. This variable can be modeled in our mathematical programming formulation (in which the $\overline{NCD}$-variables are globally minimized) from ${NCD}_i^{t \omega}$ as follows:
\begin{align*}
\overline{NCD}_i^{t \omega}  &\geq  {NCD}_i^{t \omega}\\
\overline{NCD}_i^{t \omega} &\geq 0.
\end{align*}
for all  $\omega \in \Omega, i \in N, t\in T$.
\end{itemize}

\subsection{Constraints}

The above variables are related by means of a set of linear constraints that allows to represent adequately the reallocation and sharing problem under analysis:

\begin{enumerate}
\item The product to be delivered from a node, in each period and scenario, cannot exceed a percentage of the excess of that node:
\begin{equation}\label{eq:1}\tag{${\rm C}_1$}
{\dsum_{j\in N:\atop (i,j)\in W}} x_{ij}^{t} \leq \gamma_i H_i^{t \omega}, \quad  \forall \omega \in \Omega, i \in N, t \in T.
\end{equation}
The constraint enforces that the overall amount of equipments delivered from $i$ (under scenario $\omega$) in period $t$ do not exceed the  proportion of the stock that is allowed to be delivered from the unit  to other units. It allows each unit to decide the proportion of  the excess of equipments that is willing to deliver to other units. For risk-averse units, the proportion might be small, while for risk-adverse units, the proportion might be fixed to larger amounts.
\item The amount to be delivered from a unit to other are zero unless the $y$ variables take value one and viceversa.
\begin{equation}\label{eq:2}\tag{${\rm C}_2$}
y_{ij}^t \leq x_{ij}^{t} \leq  g_i y_{ij}^{t} \quad \forall i \in N, t \in T.
\end{equation}
In case $y_{ij}^t=0$, then, one cannot deliver any product from unit $i$ to $j$, otherwise and amount between $1$ and $g_i$ can be sent.
\item Upper bound on the number of deliveries from a node  to different nodes in each period:
\begin{equation}\label{eq:3}\tag{${\rm C}_3$}
{\dsum_{j\in N:\atop (i,j)\in W}} y_{ij}^{t} \leq Q_i,  \quad \forall  {(i,j)\in W}, t \in T.
\end{equation}
As mentioned in the definition of the $y$-variables, even in case a large stock is available in unit $i$, it is not realistic to assume that such a unit deliver equipments to as much as units as desired. This constraint limits such a number to $Q_i$.
\item Amount to be shared for each group and each period:
\begin{equation}\label{eq:4}\tag{${\rm C}_4$}
\dsum_{i\in N_k} s_{ik}^{t} = q_k^t,\quad  \forall k \in P, t \in T.
\end{equation}
It is assumed that all the extra stock wants to be shared between the units in $N_k$. One may instead assume that not all the stock needs to be shared, and the constraint may be replaced by the same but with $\leq$ instead the equation.
\item Avoid to simultaneously deliver and receive equipments in the same unit:
\begin{equation}\label{eq:5}\tag{${\rm C}_5$}
 y_{ij}^{t}+y_{jk}^{t}\leq 1, \quad  \forall {(i,j),(j,k)\in W}, t \in T.
 \end{equation}
In particular, in case $k = i$, the constraint enforces that no bidirectional deliveries are allowed at the same period. Note that one may replace the set $W$ by a subset of it to allow some of the units to deliver and receive at the same period. This constraint avoids that units requiring demand deliver equipments.
 
 \item Upper bound on the storage capacity of each hospital and each period:
 \begin{equation}\label{eq:6}\tag{${\rm C}_6$}
 	H_i^{t \omega}\leq a_i, \quad  \forall \omega \in \Omega, i \in N, t \in T.
 \end{equation}
Most units may not have unlimited space to store all the equipments they receive. This constraints avoid this effects and allows receiving material only if they have space to store it or directly use it.
 \end{enumerate}
 
\subsection{Objective Functions}
 
Our stochastic mathematical programming models will have the following common shape:

\begin{align}
\min &\;\;\; \Phi(s,x,y; \Omega)\nonumber\\
\mbox{s.t. } & \eqref{eq:1}-\eqref{eq:6},\nonumber\\
&s_i^{t} \in \Z_+,\quad \forall  i \in N, t \in T,\label{p0}\tag{${\rm StochP}$}\\
&x_{ij}^{t} \in \Z_+, y_{ij}^{t} \in \{0,1\}, \quad \forall   {(i,j)\in W}, t \in T.\nonumber
\end{align}
where $\Phi(\cdot)$ will be a measure of the overall non-covered demands, and will determine the difference between the different approaches.

We consider four different robust objective functions for the mathematical programming problem described above to be minimized:
\begin{itemize}
\item \textit{Maximum non-covered expected demand of the units within the time horizon}:
$$
\Phi_1(s,x,y; \Omega) = \max_{i \in N}  \dsum_{\omega \in \Omega} p^\omega \dsum_{t\in T}  \overline{NCD}_i^{t\omega}
$$
This function equilibrates the expected non-covered demand for all the units along the whole time horizon. 
\item \textit{Maximum non-covered expected demand of the demand points at each period}:
$$
\Phi_2(s,x,y; \Omega)  = \max_{i \in N} \max_{t\in T} \dsum_{\omega \in \Omega} p^\omega \overline{NCD}_i^{t\omega}
$$
Here, one equilibrates, not only units for the whole time horizon, but also the different periods, avoiding tiny non-covered demands in a period at the price of large non-covered demands in others.
\item \textit{Maximum non-covered expected demand in each region within the time horizon}:
$$
\Phi_3(s,x,y; \Omega)  =  \max_{k=1, \ldots, L} \dsum_{\omega \in \Omega}  p^\omega\dsum_{i \in M_k} \dsum_{t\in T} \overline{NCD}_i^{t\omega}
$$
where $M_1, \ldots, M_L \subset N$ is disjoint partition of $N$ in $L$ sets. Instead of finding fair solutions for all units and periods, in $\Phi_3$, the units are aggregated by these regions, being the criterion to find equilibrate regional solutions.
 
 These sets represent different non overlapping regions in which the effect of a distribution planning wants to be measured. In practice, they can be determined by the political borders of a country (states, regions, districts, etc) in which the policies stablished at each of them want to be evaluated. 
\item \textit{Minimize the total non-covered expected demand}:
$$
\Phi_4(s,x,y; \Omega) =  \dsum_{\omega \in \Omega} p^\omega \dsum_{i \in N} \dsum_{t\in T} \  \overline{NCD}_i^{t\omega}
$$
Finally, this function account for the overall non-covered demand for all scenarios, units and period along the time horizon. 
\end{itemize}

The specific shape of the objective functions, among those described above, that is used in our decision tool must be chosen by the decision maker based on its own preferences.  $\Phi_1$ allows one to find fair solutions by units (e.g., hospitals) taking into account the most harmed ones, in terms of the non covered demand, while in $\Phi_3$, the fairness is measured by regions instead of single units. Objective $\Phi_2$ also accounts for equilibrating the non covered demand by periods, avoiding low non covered demand in some periods at the price of high non covered demands in others. Finally, $\Phi_4$ is the classical averaged measure which allows to globally minimize the non covered demand, which may harm some units to benefit others. The decision maker may also run all the models and, in view of the results, decide the most reasonable situation.

Apart from the four objective functions $\Phi_1, \ldots, \Phi_4$, we also consider in our approach their max-regret counterparts. For any $\omega \in \Omega$ we denote by $\Phi^*(\omega)$ the optimal value of the problem above but only under scenario $\omega \in \Omega$, i.e.,
\begin{align*}
\Phi^*(\omega) := \min &\;\;\; \Phi(s,x,y; \{\omega\})\\
\mbox{s.t. } & \eqref{eq:1}-\eqref{eq:6},\\
&s_i^{t} \in \Z_+,\quad \forall  i \in N, t \in T,\\
&x_{ij}^{t} \in \Z_+, y_{ij}^{t} \in \{0,1\}, \quad \forall   {(i,j)\in W}, t \in T.
\end{align*}
for any of the objective functions defined above  ($\Phi\in \{\Phi_1, \Phi_2,\Phi_3,\Phi_4\}$).

We define the regret of a solution $(x,y,s)$ under scenario $\omega \in \Omega$ as:
$$
\Phi^{\rm Regret}(s,x,y; \omega) = \Phi(s,x,y;\{\omega\}, T) - \Phi^*(\omega),
$$
that is, the difference between the actual evaluation in the global objective function $\Phi\in \{\Phi_1, \Phi_2,\Phi_3,\Phi_4\}$ of feasible solution $(s,x,y)$ under scenario $\omega$ and the optimal value obtained for such a single scenario $\omega$.

The minmax regret criterion seeks a solution minimizing the maximum regret among all scenarios, that is, it seeks a solution whose value is as close as possible to the optimal value for every scenario (see, for instance, \cite{aissi09, ben09,kas08,kou97} and the references therein). The regret version of \eqref{p0} is:

\begin{align}
\displaystyle \min \max_{\omega \in \Omega} &\;\;\; \Phi^{\rm Regret}(s,x,y; \omega)\nonumber\\
\mbox{s.t. } & \eqref{eq:1}-\eqref{eq:6},\nonumber\\
&s_i^{t} \in \Z_+,\quad \forall  i \in N, t \in T,\label{p1}\tag{RegretP$_0$}\\
&x_{ij}^{t} \in \Z_+, y_{ij}^{t} \in \{0,1\}, \quad \forall   {(i,j)\in W}, t \in T.\nonumber
\end{align}

The above formulation can be equivalently rewritten as:
\begin{align}
\displaystyle \min  &\;\;\; \alpha \nonumber\\
\mbox{s.t. } & \alpha \geq \Phi(s,x,y;\{\omega\}, T) - \Phi^*(\omega), \quad \forall \omega \in \Omega,\nonumber\\
&\eqref{eq:1}-\eqref{eq:6},\nonumber\\
&s_i^{t} \in \Z_+,\quad \forall  i \in N, t \in T,\label{p2}\tag{RegretP}\\
&x_{ij}^{t} \in \Z_+, y_{ij}^{t} \in \{0,1\}, \quad \forall   {(i,j)\in W}, t \in T,\nonumber
\end{align}
for each  $\Phi \in \{\Phi_1, \Phi_2,\Phi_3,\Phi_4\}$, resulting in four alternative objective functions $\Phi_1^{\rm Regret}$, $\Phi_2^{\rm Regret}$, $\Phi_3^{\rm Regret}$, $\Phi_4^{\rm Regret}$.


\section{Math-Heuristic Procedure}\label{sec:4}

The mathematical programming formulations described in Section \ref{sec:3} involve integer variables to represent the decision variables concerning the deliveries ($x$) and the shared amounts ($s$). Furthermore, they use other sets of auxiliary binary and continuous variables in order to adequately represent the constraints and the objective functions, as $y$ or those that allows to model the excess of stock or the nonnegative non-covered demand. Thus, the Mixed Integer Linear Programming (MILP) model becomes hard to solve when the number of units ($N$) and periods ($T$) is large, as usual. In this section we describe a math-heuristic approach that allows us to obtain good quality feasible solutions for the problem in reasonable computational times, but still using mathematical programming tools to solve up to optimality some subproblems. The main idea under the heuristic is to split the time horizon in shorter time horizons and \emph{merge} the obtained results of the smaller problems adequately. 

In our approach we split $T=\{1, \ldots, q\}$ into smaller sorted non-overlapping subperiods, $T_1, \ldots, T_K$ with $T_k = \{t_{k-1}, t_{k-1}+1, \ldots, t_{k}-1\}$ for $k=1, \ldots, K$, where $1=:t_{0} < t_{1} < \cdots < t_{K}:=q$. Although one may solve the problem at each of the sets $T_k$ instead of on the whole $T$, the obtained solution is not feasible for our problem since the initial stock at the beginning of each subperiod is not defined, except for the first interval. To overcome this difficulty we propose an approach to adequately \textit{glue} the obtained solutions to construct a feasible solution of the original problem. For the sake of this gluing process, instead of solving the problems within the time periods $T_1, \ldots, T_K$, we consider the subperiods $T_1^+, \ldots, T_K^+$, where $T_k^+ = T_k \cup \{t_k\}$, i.e., the first element in $T_{k+1}$ is added to $T_k$ (except for $k=K$) in order to link them with a single common period. Next, the MILP is solved for $T_1^+$, calculating the reallocation and sharing policies for that interval, for all its time periods instead those of the last one, $t_2$, where, instead, we compute the excess of each unit at that period. This excess is used as input of the initial stock for solving the next subperiod, $T_2^+$. The process is repeated until the complete execution of all the subperiods is performed. Observe that, unless delivering times are zero, the $x$ and $y$ variables in the last period of each interval are zero, since it is not possible to cover any demand. The values of the $x$ and $s$-variables are the sequentially obtained, while the value of the objective function (for each of them) has to be constructed once the procedure is terminated. In Algorithm \ref{alg} we show the pseudocode of proposed procedure. 
\begin{algorithm}[h]
\SetAlgoLined


Choose $t_{0}, t_{1}, \ldots, t_{K}$ such that $1=:t_{0} < t_{1} < \cdots < t_{K}:=q$.

Define:
\begin{itemize}
\item $T_{k} = \{t_{k-1}, t_{k-1}+1, \ldots, t_{k}-1\}$ for $k=1, \ldots, K$,
\item $T_k^+ = T_k \cup \{t_{k}\}$ for $k=1, \ldots, K-1$.
\item $T_K^+=T_K$.
\end{itemize}

 \For{$k=1, \ldots, K$}{
\begin{enumerate}
\item Solve:
\begin{align}
\hspace*{-2.3cm}\min & \hspace*{0.1cm}\Phi(s,x,y; \Omega(T_k^+))\nonumber\\
\hspace*{-2.3cm}\mbox{s.t. } & \eqref{eq:1}-\eqref{eq:6},\nonumber\\
\hspace*{-2.3cm}&s_i^{t} \in \Z_+,\quad \forall  i \in N, t \in T_k^+,\label{pk}\tag{${\rm StochP}(T_k^+)$}\\
\hspace*{-2.3cm}&x_{ij}^{t} \in \Z_+, y_{ij}^{t} \in \{0,1\}, \quad \forall   {(i,j)\in W}, t \in T_k^+.\nonumber
\end{align}
\item Set $\bar s^t_{i_0l}$, $\bar x_{ij}^t$ and $\bar y_{ij}^t$, for $i_0 \in N_l, (i,j) \in W, l \in P$, $t\in T_k$ to the optimal values of \eqref{pk}.
\item Set $s_i^0(k)$ to  $S_i^{t_k} + R_i^{t_k} - D_i^{(t_k-1) }$ for $i\in N$.
\end{enumerate}}
\KwOut{$(\bar s, \bar x, \bar y)$}

 \caption{Heuristic prodedure.\label{alg}}
\end{algorithm}

\begin{theorem}
Algorithm \ref{alg} provides a feasible solution for \eqref{p0}.
\end{theorem}
\begin{proof}
Observe that at each time period, $t$, the obtained solution verifies all the constraints of the model, except \eqref{eq:1} and \eqref{eq:6} which depend on the auxiliary $H$-variables, since they are separable by the index $t$.  For the case of \eqref{eq:1} and \eqref{eq:6}, they depend on the $H$-variables which were defined as:
$$
H_i^{t \omega} = \max\{ 0, S_i^{t } + R_i^{t } - D_i^{(t-1) } -d_i^{t \omega} \},
$$
Thus, they depend on the amount of product received and delivered until the previous period, and also on the stock accumulated until that period (which are accumulated from the initial period to period $t$). Let us denote by:
$$
h_i^{t}(k) =S_i^{t}(k) + R_i^{t }(k) - D_i^{(t-1) }(k), \;\;\forall i \in N, \omega \in \Omega,
$$
the amount of available product in unit $i$ at period $t\in T_k$ before attending the demand of the unit, but after receiving (from other units or shared by its region) equipments and delivering, where 
\begin{eqnarray*}
S_i^{t}(k) &= s_i^0(k) + \dsum_{j\in P:\atop i \in N_j} \dsum_{t^\prime \leq t} s_{ij}^{t^\prime},\\
R_i^{t }(k) &= {\dsum_{j\in N:\atop (j,i)\in W}} \dsum_{t_{k-1} \leq t^\prime \leq t:\atop t^\prime  + \ell_{ji}\leq t} x_{ji}^{t^\prime },\\
D_{i}^{t}(k) &=  {\dsum_{j\in N:\atop (i,j)\in W}}\dsum_{t_{k-1}\leq t^\prime\leq t} x_{ij}^{t^\prime}.
\end{eqnarray*}
for $k \in \{1, \ldots, K\}$, $i\in N$ and $t \in T_k$.

Note that this values are the subperiod counterparts of the global variables defined in our model. 

For $t \in T_k$ we get that:
\begin{eqnarray*}
\begin{split}
h_i^{t}(k) &= S_i^{t_k} + R_i^{t_k} - D_i^{(t_k-1) } + \sum_{j\in P:\atop i \in N_j} \dsum_{t^\prime \leq t} s_{ij}^{t^\prime} + {\dsum_{j\in N:\atop (j,i)\in W}} \dsum_{t_{k-1} \leq t^\prime \leq t:\atop t^\prime  + \ell_{ji}\leq t} x_{ji}^{t^\prime } -  {\dsum_{j\in N:\atop (i,j)\in W}}\dsum_{t_{k-1}\leq t^\prime\leq t-1} x_{ij}^{t^\prime}\\
&= S_i^t + R_i^t - D_i^t
\end{split}
\end{eqnarray*}
Thus, $H_i^{t\omega} = \max\{0, h_i^{t}(k)-d_i^{t\omega}\}$ if $t \in T_k$. Thus, the $H$-variables in our model are adequately recovered from the solutions obtained solving the problem by subperiods. Since for each of the subproblems on $T_k^+$, the constraints:
$$
\dsum_{j\in N:\atop (i,j)\in W} x_{ij}^{t} \leq \gamma_i \max\{0, h_i^t(k)-d_i^{t\omega}\} \text{ and } H_i^{t \omega}\leq a_i
$$
are verified for $t \in T_k$ and $i \in N$, $\omega \in \Omega$ and $k \in \{1, \ldots, K\}$, then \eqref{eq:1} and \eqref{eq:6} are also verified.

\end{proof}

From the above result, we get that our procedure provides feasible solution to our problem, and then gives us upper bounds for the exact optimal values of our problems.

\section{Case Study: Reallocation and Sharing of Ventilators in Spain}\label{sec:5}

One of the main causes for the critical situation in hospitals during the COVID-19 crisis in Spain has been the high demand of invasive mechanical ventilation among severe patients, together with the lack of this resource in some hospitals around the country. Invasive mechanical ventilation is used to assist patients with serious breathing problems (\cite{invasive,mechanical}). 

We devote this section to analyze the reallocation and sharing of invasive mechanical ventilators (from now on ventilators) in two regions in Spain with different demand distribution during the first wave: the region Madrid, in which the pandemic caused a large amount of critical patients and dead, and the region of Andaluc\'ia, in which the situation was slightly less critical. We use the proposed mathematical programming models to determine if the covered demand of patients needing a ventilator could have been significantly improved if reallocation and sharing would have been applied during the first COVID-19 wave, the $49$ days from March 8th to April 25th, 2020.

\subsection{Input Information}

In what follows we describe the input information that we use in our models as well as the results obtained after running them. All the input information that we use in our experiments are available in the GitHub repository \url{https://github.com/vblancoOR/RedistributionCOVID19}.

\subsection*{Graph Structure}

We consider two types of graph structures trying to simulate possible real networks of the regions of Madrid and Andaluc\'ia. While the region of Madrid has 51 hospitals, the region of Andaluc\'ia has 106. 

Note that being the graph structure an input, it can be modified to adjust the reality of the regions. In particular:
\begin{description}
\item[Complete (C):] We consider a complete graph in which all nodes (units) are connected bidirectionally. 
\item[Logistic Centers (LC):] We incorporate logistic centers of provinces and regions, and we consider that each hospital of a province is only  bidirectionally linked with the logistic center of the province, and the logistic centers of the provinces are linked through the regional logistic center. For each of the two considered regions the situation is different:
\begin{itemize}
\item \textit{Region of Madrid}: Since this region has a single province, we assume that the unique logistic center is located in the city of Madrid in \textit{Hospital de la Paz}. 
\item \textit{Region of Andaluc\'ia}: In this region there exist a logistic center in each of its eight provinces. We assume that the regional logistic center is located in the \textit{Hospital de Antequera} in M\'alaga (geographical center of Andaluc\'ia). This graph has 105 arcs. In Figure \ref{Graph_andalucia}, we show, this graph structure in this region. 
\begin{figure}
	\begin{center}
		\input{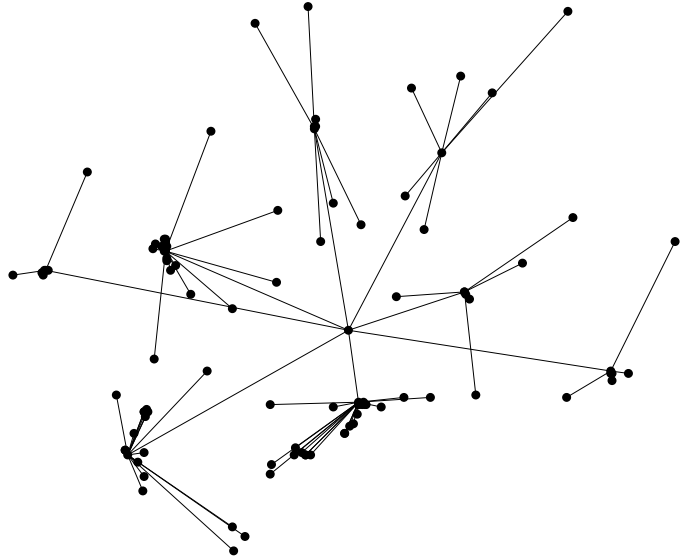}
    \end{center}
	\caption{Graph structure with logistics centers in Andaluc\'ia.}\label{Graph_andalucia}
\end{figure}
\end{itemize}
\end{description}

For the region of Madrid a single $N_1$ region is considered (containing all the 51 hospitals in the region). In Andaluc\'ia, we consider $N_1, \ldots, N_8$ as the sets containing each of the hospitals in the different provinces. The sets $M$ for the objective functions $\Phi_3$ and $\Phi^{\rm Regret}_3$ coincide with the $N$-sets, that is, the regions for which fair solutions are desired are the same as those for which the sharing policies are designed.

Furthermore, in Andaluc\'ia, for the LC-graph structure, we also use the set $N_9$ containing the eight logistic centers (but it is not included in the $M$-sets used in the $\Phi_3$-objective functions).

\subsection*{Initial stocks}

The initial number of invasive mechanical ventilators in each hospital is crucial to conduct an accurate study. However, these official stocks were not publicly available in Spain. In order to estimate them we collected  information from different sources. Since the main reason because patients are admitted to an ICU is to receive ventilatory support \cite{invasive}, we assume that this number coincides with the number of ICU beds, although in some cases this amounts can be slightly larger due to extra ventilators situated at other types of beds. This underestimation may be favorable since those ventilators are available for the hospitals in case more patients than the estimated need a ventilator. The proportion of ICU beds in public and private hospitals of each region~\cite{datadista} together with the the number of beds in each hospital~\cite{CNH} allows us to estimate the initial stock of ventilators of each hospital.

\subsection*{Extra stock}

The extra stock at each time period indicates the new available ventilators to share (if any) in that period among the hospitals of a given region. During the COVID-19 crisis different national or regional governments have bought  and received extra invasive mechanical ventilators as reported by some national newspapers (\cite{22andalucia,10extremadura,9murcia,amancio,11galicia,351madrid,espresp},  among many others). The situation in the two considered regions is different:
\begin{itemize}
\item {\it Region of Madrid:} In this region, the regional government received 351 ventilators at the end of March  and 213 ventilators at the early April~\cite{351madrid}. It also received ventilators from other regions of Spain: Galicia lent 11 ventilators \cite{11galicia} on March 27th, Andaluc\'ia 22 ventilators \cite{22andalucia} at the end of March, Extremadura 10 ventilators \cite{10extremadura} at the end of March and Murcia 9 ventilators \cite{9murcia} in early April.
\item {\it Region of Andaluc\'ia:} In Andaluc\'ia, there is no public information about the extra stock. However, a significant issue in this region is that Andaluc\'ia started to manufacture its own ventilators under the project \textit{Andaluc\'ia Respira} \cite{andresp1,andresp2}. Although these ventilators fulfill the quality requirements established by the Ministry of Health, we did not find any information stating that they have been distributed by the region yet.
\end{itemize}

Apart from the above, some extra stock has also been provided from the Spanish Government and private donations. A total of 2400 ventilators \cite{amancio,espresp} was received to share among all the regions; but there is no information about when and where the ventilators were allocated. We assume that these ventilators were distributed among the regions by means of population. Hence, in the cases of Madrid and Andaluc\'ia, we estimate that they received, in the second week of April, 340 and 429 ventilators, respectively.

\subsection*{Capacities}

\begin{itemize}
\item $Q_i$: The maximum number of deliveries from a hospital depends on the graph structure. For the C-Graph we set it to $5$, to avoid extra work on preparing packages of ventilators for different trucks. For LC-graph, the parameter for the logistic centers was fixed to infinity, while for the hospitals, we set it to $0.4$ times the number of adjacent nodes to the logistic center.
\item $\gamma_i$: The percentage of excess that can be delivered from a hospital was fixed to $0.8$.
\item $g_i$: The amount of ventilators to be delivered by each hospitals is set to 20, as a measure of transport  capacity.
\item $a_i$: The storage limit for the hospitals is set to twice the number of ICU beds in the hospital. For logistic centers the parameter is set to infinity.
\end{itemize}

\subsection*{Shipping times}

The arc weights of the graphs are defined based on the geographical distance between the nodes. Then, we compute the shortest paths between each pair of nodes (in each graph structure), and we relativize the transportation times to the  largest one. Apart of that we include in the shipping times different graph-dependent processing times: one day for the C-graph and 0.1 days per used logistic center in the LC-graph. 

\subsection*{Demands}

We estimated the daily ventilator demands based on real ICU demands of COVID-19 patients in Madrid and Andaluc\'ia, published by the Spanish \cite{spaindata} and Andaluc\'ia \cite{andaluciadata} governments from 08/03/2020 to 25/04/2020 (first COVID-19 wave data).

This data was collected differently in each region. The region of Madrid reported the daily ICU demands, and then, ready to be incorporated to the models. On the other hand, the government of Andaluc\'ia  reported the accumulated demand of ICU. We estimated the daily demands as follows: for each hospital,  we compute the daily number or new COVID-19 patients in ICU and we assume that each of them stays in ICU 21 days (the average number of days the patients stay on a ICU bed during the first wave~\cite{21dias}). Scenario ``\texttt{Real}'' was created with these estimated demands. This scenario can be considered as the closest to the real situation that Spain lived during the first wave. 

Using these demands, we randomly generate two more scenarios as follows:
\begin{itemize}
\item Choose $r_j^-, r_j^+$ uniformly distributed in $[0,0.5]$, for each province $j \in P$. 
\item Choose $\tilde{r}_i^-$ (resp. $\tilde{r}_i^+$) uniformly distributed in $[0,r_j^-]$ (resp. $[0,r_j^+]$), for each hospital $i$ in the province, $j$, and set, for each scenario the following demands:
\begin{description}
\item[Scenario ``\texttt{Pessimistic}'':] 
$
	\tilde{d}_i^t =(1+ \tilde{r}_i^+) d_i^t, \quad i\in N_j, j \in P, t\in T.
$
\item[Scenario ``\texttt{Optimistic}'':] 
$
\tilde{d}_i^t = (1- \tilde{r}_i^-) d_i^t, \quad i\in N_j, j \in P, t\in T.
$

\end{description}
In case $\tilde{d}_i^t\leq 0$ we set random integer value  in $\{1,2\}$.

\end{itemize}
In Figure \ref{demands}, we show the demands on Scenario \texttt{Real} for the two regions.

\begin{figure}
\begin{center}
	\includegraphics[scale=0.3]{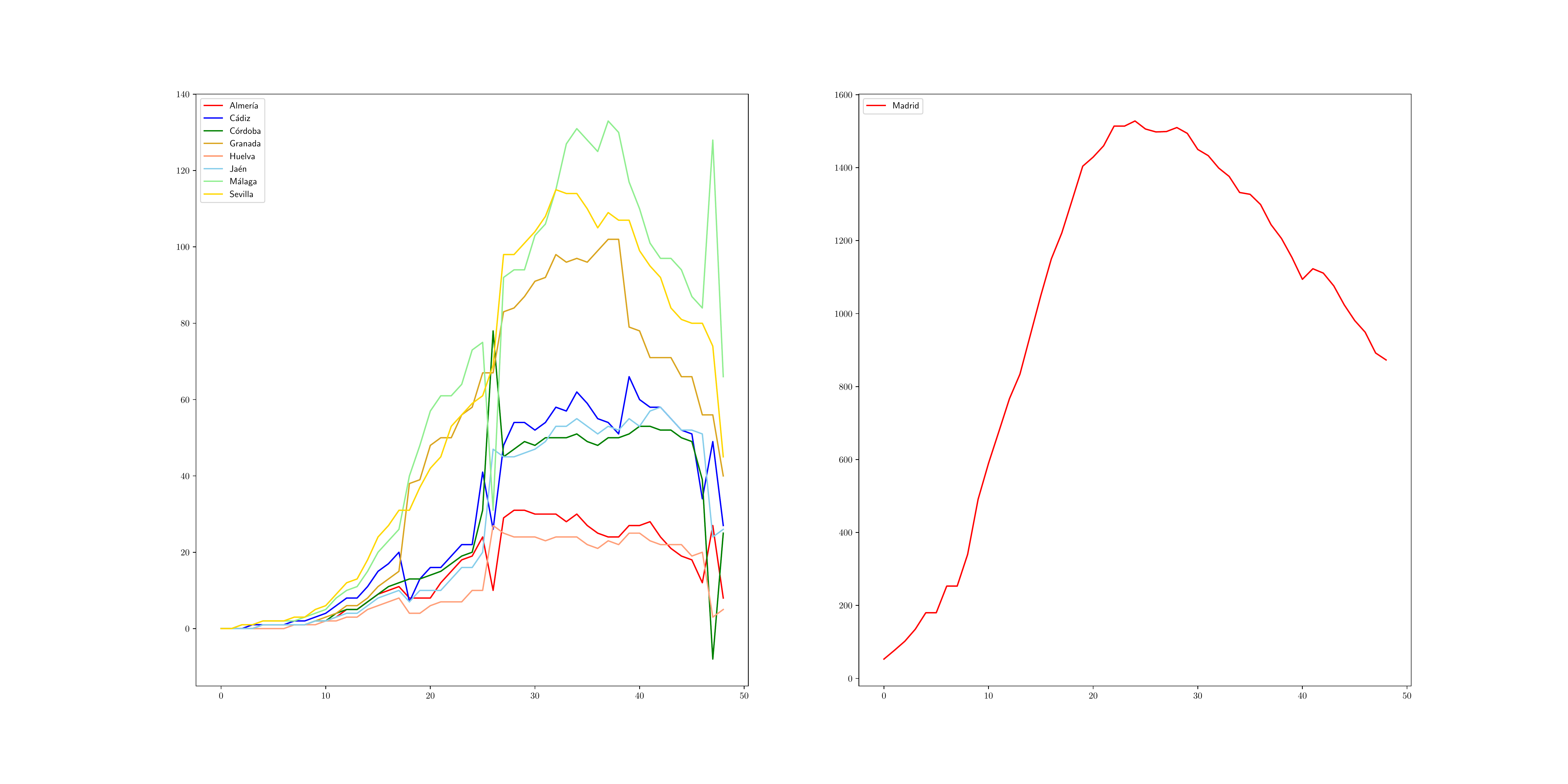}
	\caption{\texttt{Real} demands of ventilators for Region of Andaluc\'ia (left) and Region of Madrid (right).}\label{demands}
\end{center}
\end{figure}

\subsection{Results}

The models have been coded in Python 3.7 and using as optimization solver Gurobi 9.0 in a MacBook Pro with a Core 2 Duo CPU clocked at 2,66 GHz and 4GB of RAM memory. We have run our math-heuristic algorithm partitioning the time horizon into $12$ subperiods. A time limit of 1 hour was fixed for solving the subproblems, although none of our models reached such a limit.

We have applied our approach to the regions of Madrid and Andaluc\'ia. Moreover, we have also run the models in case no sharing is allowed between the hospital. The aim is to compare the obtained redistribution and sharing policies to the real situation in Spain, in which redistribution was not  implemented. However, since there is no information about the sharing policies of the extra stock in the real situation, we assume that it was performed \textit{optimally} according to our models (only fixing in them to $0$ the $x$-variables). Note that we compare our redistribution and sharing proposal to a situation which is better than the actually implemented.

In the following sections, we analyze the results and conclusions obtained through the numerical study. We illustrate them with different graphics and figures for particular configurations of scenarios, objective functions or types of graphs. However, the rest of the figures for the remaining configurations of scenarios, functions or graphs can be found in \url{https://github.com/vblancoOR/RedistributionCOVID19} for the interested readers.

\subsubsection*{Redistribution vs. No Redistribution}

We start by comparing the non-covered demand, that is, the number of patients needing a ventilator that were not attended due to the lack of this resource, if the proposed redistribution is carried out, or not. In Figures \ref{fig:RvsNRMad} and \ref{fig:RvsNRAnd}, in each of the graphics, the continuous red line represents the total non-covered demand at each time period if  the \textit{real} scenario occurs and redistribution is allowed. The dashed red line shows the total non-covered demand at each time period, if the \textit{real} scenario happens and only the redistribution of the extra stock is allowed, but not the sharing of available stock. Remind that this second case represents a situation better than what was actually applied in Spain, since in this case, the redistribution of the extra stock is done optimally. For simplicity, we will refer to this second case as the \textit{without redistribution} case. The continuous and dashed green lines represent, respectively, the total available stock in the two described situations: with or without redistribution. Notice that we solve the models taking into account that any of the three considered scenarios can occur, but we represent in these figures the actual behavior if the obtained solution is implemented when the \textit{real} scenario happens.

We show in Figure \ref{fig:RvsNRMad} the case of the Madrid region, for the graph with logistic centers and objective $\Phi_2^{\rm Regret}$ in the left graphic, and for the complete graph and objective $\Phi_1$ in the right one.  We can observe that in all the cases, when redistribution is not carried out, there exist always available stock but also demand that is not being covered. For instance, in the second graphic we can see that in period $25$, there are around $250$ available ventilators but around $650$ patients that are not attended. However, when redistribution is considered, in the periods in which the non-covered demand is positive, the available stock is almost zero, that is, the available stock is redistributed and used. This implies a significant decrease in the number of non-treated patients. For instance, in the same case described before, the non-covered demand reduces to less that $470$. A similar behaviour can be observed in the rest of objectives, graphs and scenarios. 

In Figure \ref{fig:RvsNRAnd}, we show the case of the Andalucía region, for a graph with logistic centers and objective $\Phi_3^{\rm Regret}$ (right), and for the complete graph and objective function $\Phi_4$ (left). In this case, we can observe that most of the demand is covered in both cases, with and without redistribution. The reason for this is the availability of stock to cover all the demand in most of the periods. However, there is a critical period, between $t=25$ and $t=33$, in which the demand increases (see Figure \ref{demands}) and the non-covered demand is positive for the case without redistribution, even thought there exist more than $200$ available ventilators. It can also be observed that for some objective functions, the non-covered demand in the redistribution case is also positive despite existing available stock. The reason for this can be twofold: 1) Since the problem is solved using the heuristic described in Algorithm \ref{alg} in which the total time period is split into smaller periods, the heuristic is not able to anticipate, for all the objective functions, the rapid increase of the demand after period $24$. 2). The collected information for the provinces of Málaga and Sevilla is not accurate, with a sudden big increase and decrease in period $27$, as can be observed in Figure \ref{demands}, which can not always be efficiently handled. However, in most of the cases, the total non-covered demand is lower for the redistribution case, being this value even zero for some of the objective functions, as it is the case of $\Phi_4$. 

We conclude therefore that redistributing the available stock significantly increases the number of treated patients.

\begin{figure}[h]
\begin{center}
\hspace*{-0.85cm}\includegraphics[scale=0.3]{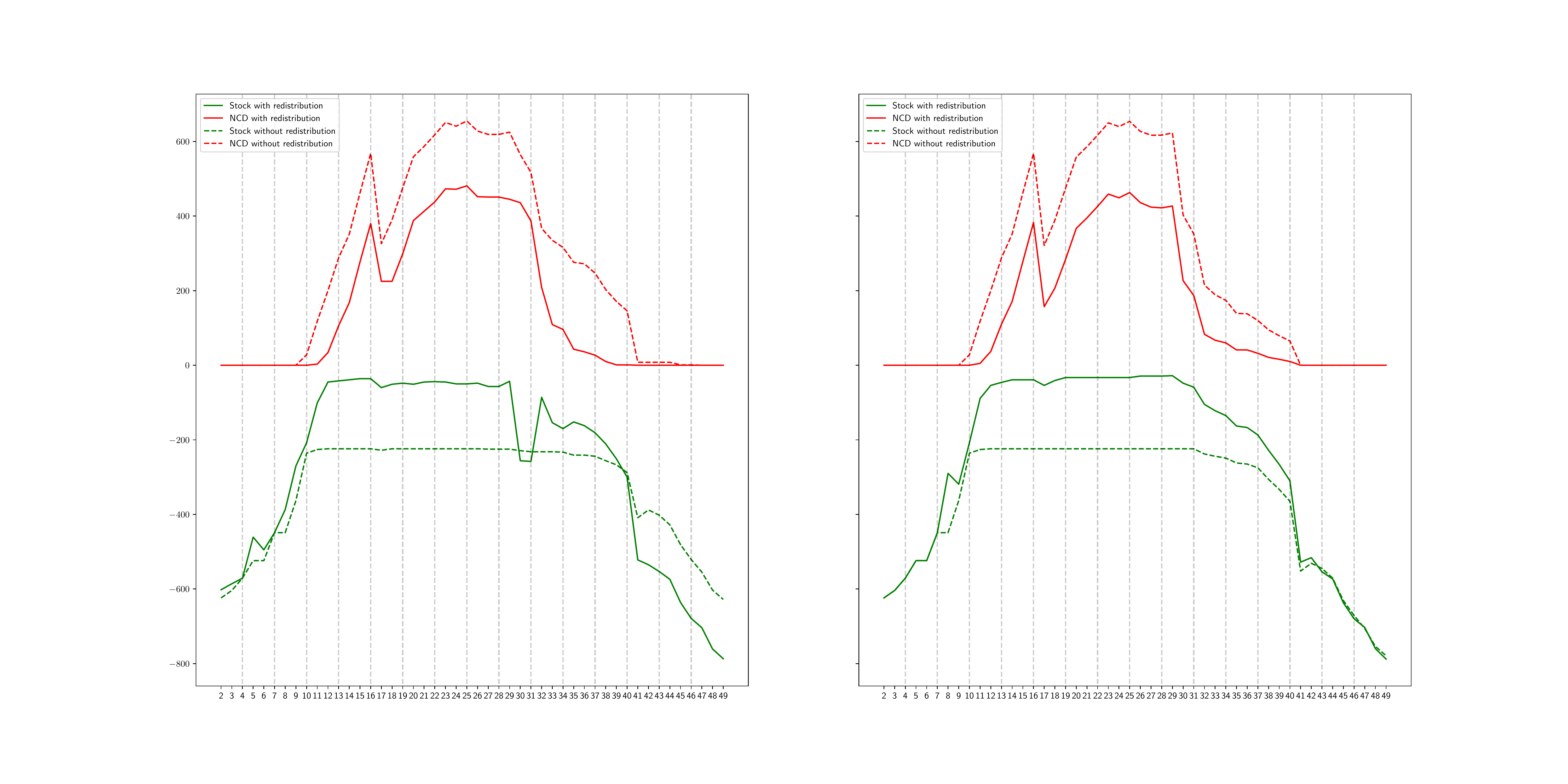}
\end{center}
\caption{Non-covered demand (red lines) and available stock (green lines) at each time period, if the \textit{real} scenario occurs, with (continuous line) and without (dashed line) redistribution, in Madrid, for LC-graph and objective $\Phi_2^{\rm Regret}$ (left) and for C-graph and objective $\Phi_1$(right). } \label{fig:RvsNRMad}
\end{figure} 

\begin{figure}[h]
\hspace*{-0.85cm}\includegraphics[scale=0.3]{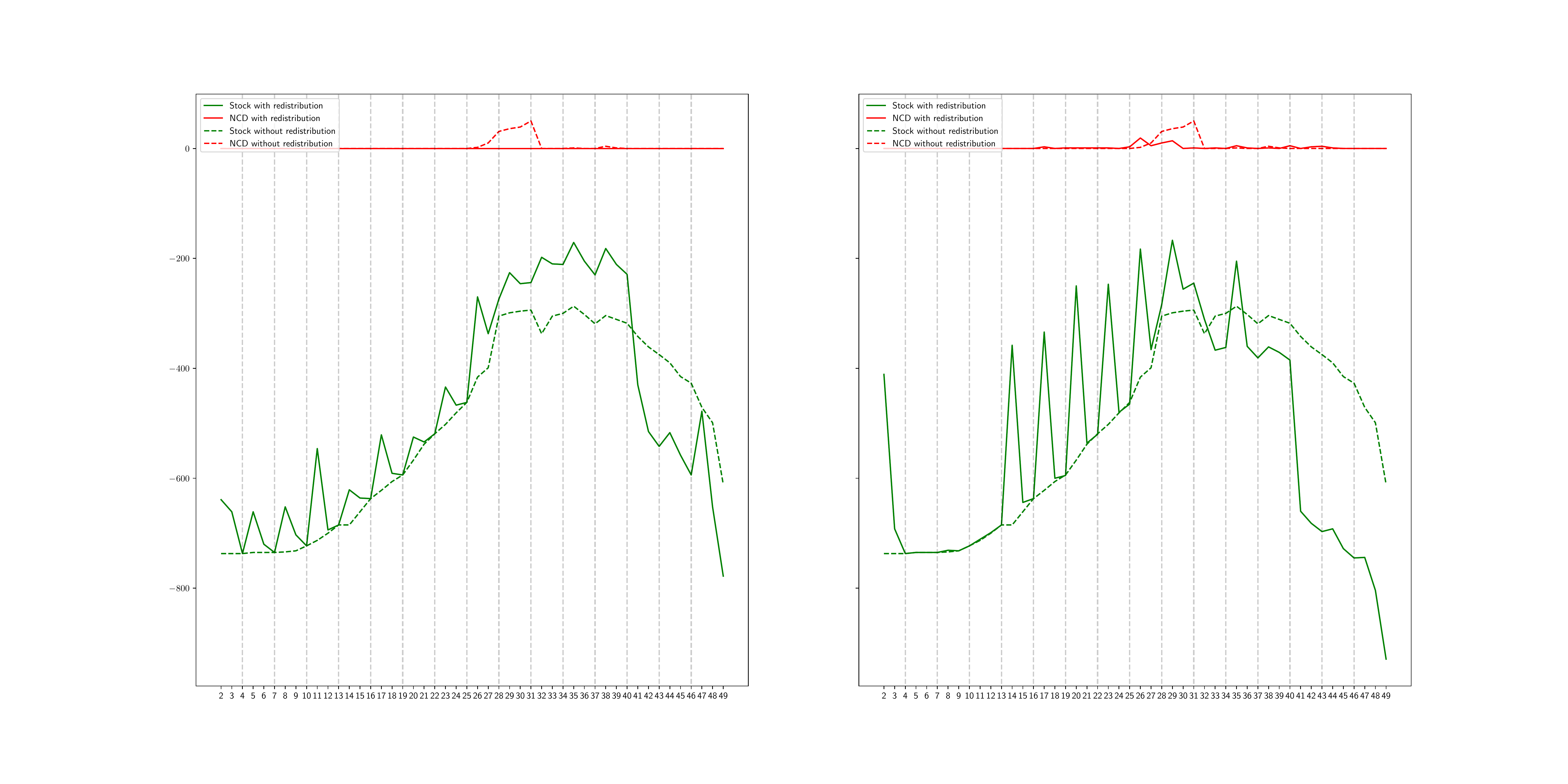}

\caption{Non-covered demand (red lines) and available stock (green lines) at each time period, if the \textit{real} scenario occurs, with (continuous line) and without (dashed line) redistribution, in Andalucía, for LC-graph and objective function $\Phi_4$ (left) and for C-graph and objective function $\Phi_3^{\rm Regret}$(right).}\label{fig:RvsNRAnd}
\end{figure}

\subsubsection*{Comparison of Redistribution and Sharing Policies}$ $\\

We compare in this section the behavior of the redistribution through the time horizon and provinces.

We show in Figure \ref{fig:barras1}, the number of redistributed ventilators through the LC-graphs in the achieved solution, at each time period, for the region of Madrid (left) and Andalucía (right) for objective function $\Phi_3$. We also include in such graphics the demand curves for the three scenarios.  Figure \ref{fig:barras2} shows the same but for the case of C-graphs and objective function $\Phi_3^{\rm Regret}$. We can observe that the amount of redistributed stock is much higher in the region of Andalucía than in the region of Madrid. This is caused by the fact that the number of demanded ventilators in Madrid is much greater than such number in Andalucía. The high demand in most of the hospitals in Madrid make practically nonexistent the availability of stock to share. The existence of stock to redistribute in periods 6 to 10 in Madrid in is due to the lower demand, and the availability of stock in periods 16, 30 and 32 in Madrid in Figure \ref{fig:barras1} is due to the entrance of a high quantity of extra stock: 351, 213 and 115 ventilators, respectively, to the logistic center that it is later redistributed. This effect is not observed in the case of the C-graphs, Figure \ref{fig:barras2}, because here the extra stock is directly distributed among the hospitals, and in this graphic we only show the redistribution among hospital, not the distribution of the extra stock. In the case of Andalucía, the highest amounts of shared ventilators coincide with the periods in which there is lower demand, and therefore more available stock to redistribute anticipating future increases in the demands. In the case of the graphs including logistic centers, there exist more periods with redistribution due to the same effect explained for the case of Madrid, and also to the fact that the delivery constraints imposed to the logistic are less restrictive than for the hospitals.

Note that we are optimizing the non-covered demand, hence, the management of the available stock could maybe have been done better.

\begin{figure}[H]
\hspace*{-0.75cm}\includegraphics[scale=0.3]{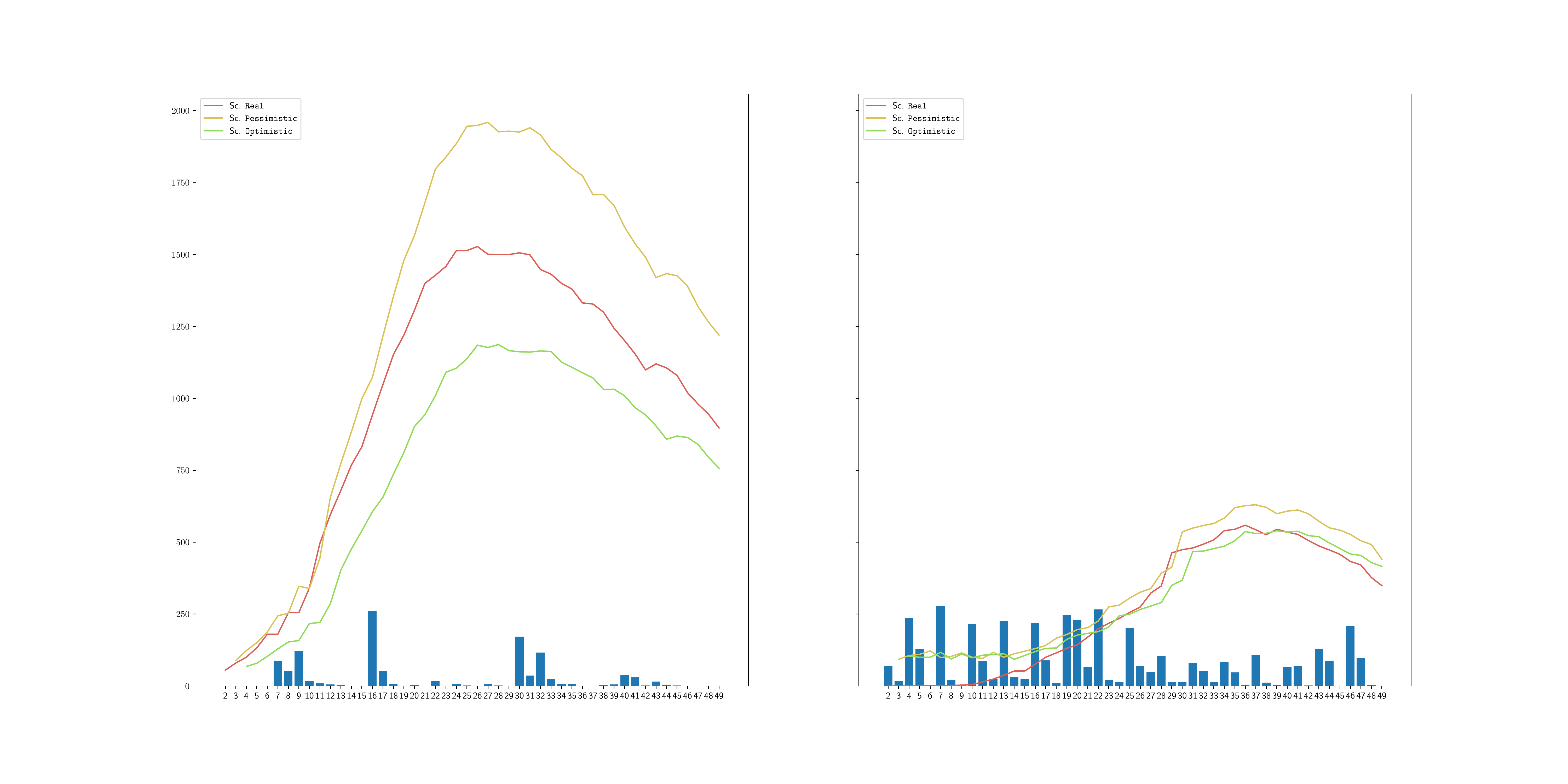}

\caption{Amount of redistributed stock throught the LC-graphs, at each time period, for the objective function $\Phi_3$, for the region of Madrid (left) and the region of Andalucía (right). } \label{fig:barras1} 
\end{figure}

\begin{figure}[H]
\hspace*{-0.75cm}\includegraphics[scale=0.3]{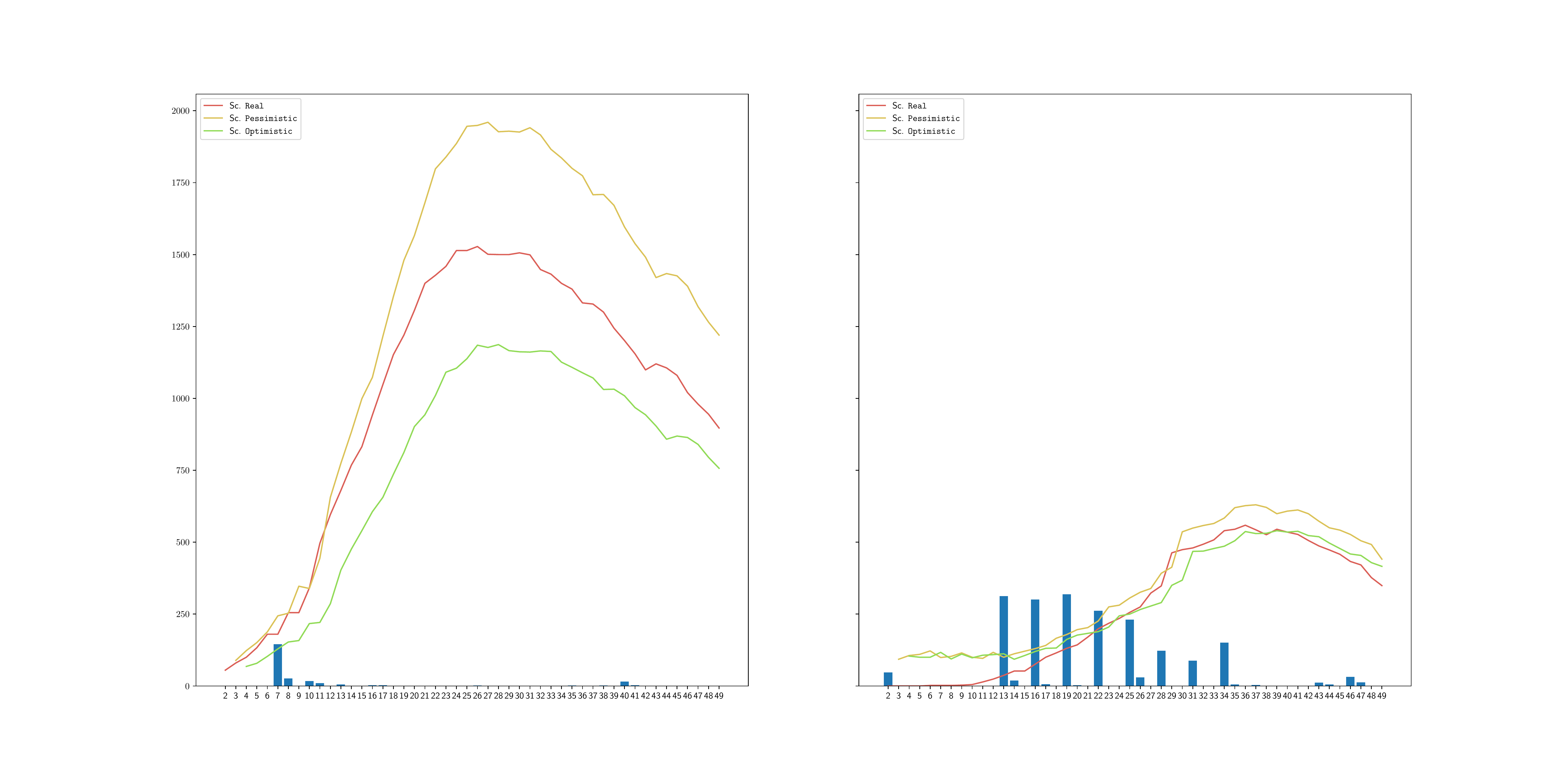}

\caption{Amount of redistributed stock throught the C-graphs, at each time period, for the objective function $\Phi_3^{\rm Regret}$, for the region of Madrid (left) and the region of Andalucía (right). } \label{fig:barras2} 
\end{figure} 

We include in the following Figure \ref{fig:tarta}, three different pie charts. The first one contains the distribution by proportion of population of the regions in Andalucía. In the remaining charts, we show the proportion of extra stock in the obtained solution for the model with redistribution and objective functions $\Phi_4$,  and $\Phi_4^{\rm Regret}$, respectively, for LC-graphs in the first case, and C-graphs in the last case. Note that the share of extra stock is performed in our model based on the demand required by each hospital and not on population, and then, the pie charts reflect that our model allocates the extra stock by demand. We can observe for instance in the diagram for $\Phi_4$ and LC-graphs, the one in the center, that most of the extra stock is allocated in the province of Córdoba, which is not concentrating the highest proportions of population. For the case of the C-graph and objective function $\Phi_4^{\rm Regret}$, the distribution is not mostly concentrated in a single province, it is withal more divided among different provinces. However, also in this case, the distribution is not carried out according to the proportion of inhabitants. For instance, we can see that despite being Sevilla the province with the highest proportion of inhabitants in Andalucía, the proportion of extra stock sent to this province (grey) is one of the lowest ones.

In these pie charts we can also observed the previously described effect: the extra stock is more distributed into different provinces in the the case of the C-graphs that in the case of the LC-graphs. This is due to the fact that it is easier to redistribute later the available stock from the logistic centers than from the hospitals, and also to the fact that the logistic centers have a higher capacity to store stock.

\begin{figure}[H]
\begin{center}
\includegraphics[scale=0.65]{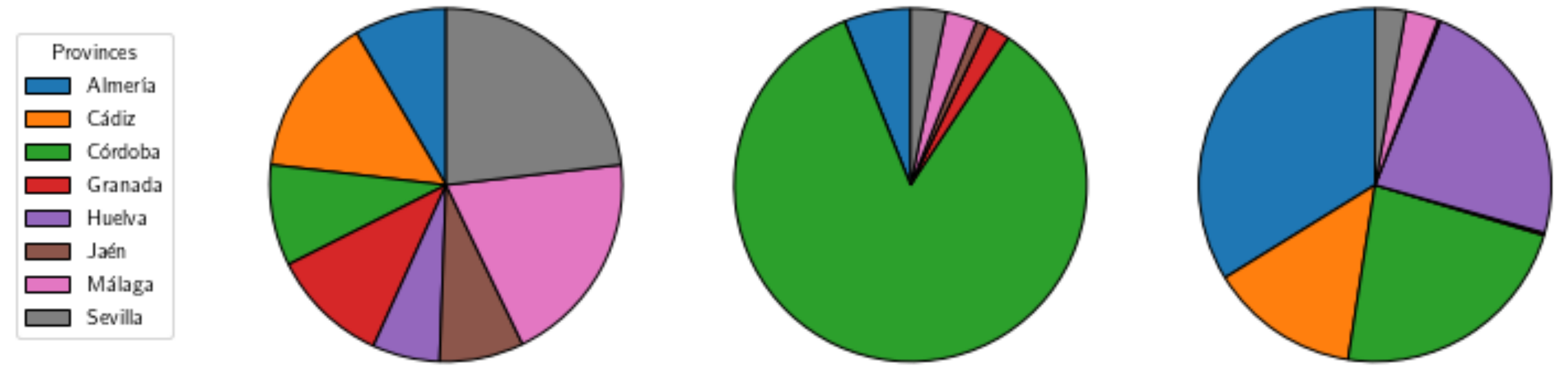}
\end{center}
\caption{Distribution of proportion of extra stock in the provinces of Andaluc\'ia: by population (left), by our model for LC-graph and $\Phi_4$ (center), and C-graph and $\Phi_{4}^{\rm Regret}$(right).}\label{fig:tarta}
\end{figure}

\subsubsection*{Comparison of Non-covered Demand by Scenarios}

We show in this section, in Figures \ref{fig:ScenariosMadrid} and \ref{fig:ScenariosAnd}, the behaviour of the non-covered demand if the achieved solution is implemented in each of the three considered scenarios: optimistic (left), real (center) and pessimistic (right). Each graphic in each of these figures follows the same style that the graphics presented in Figure \ref{fig:RvsNRMad}. Figure \ref{fig:ScenariosMadrid} illustrates the case of Madrid and objective function $\Phi_1^{\rm Regret}$, and Figure \ref{fig:ScenariosAnd} the case of Andalucía and objective function $\Phi_4$. In both cases we can appreciate that the pattern of the non-covered demand is practically the same for the three different scenarios, since the tendency of the demand is the same in the three scenarios, but, as expected, the higher the amount of demand, the higher the non-covered demand. For instance, in the worst moment in Madrid and our approach is applied, if optimistic scenario occurs, the non-covered demand aroun $250$, meanwhile if pessimistic scenario occurs, the non-covered demand is more than $750$. For the case without redistribution, this amounts are around $300$ and more than $1000$.

\begin{figure}[h]
 \hspace*{-1.75cm}\includegraphics[scale=0.22]{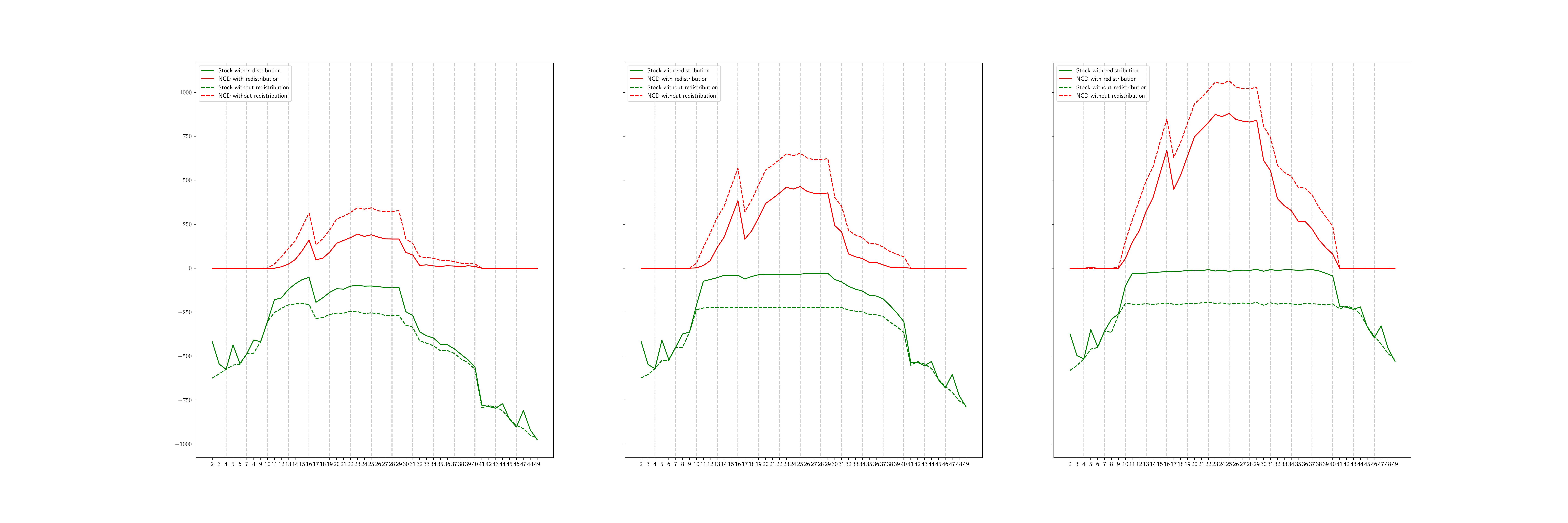}
\caption{Non-covered demand (red lines) and available stock (green lines) at each time period, if the \textit{optimistic} (left), \textit{real} (center) or \textit{pessimistic} (left) scenario occurs, with (continuous line) and without (dashed line) redistribution, in Madrid, for C-graph and objective function $\Phi_1^{\rm Regret}$.}\label{fig:ScenariosMadrid}
\end{figure}

\begin{figure}[h]
\hspace*{-1.75cm} \includegraphics[scale=0.22]{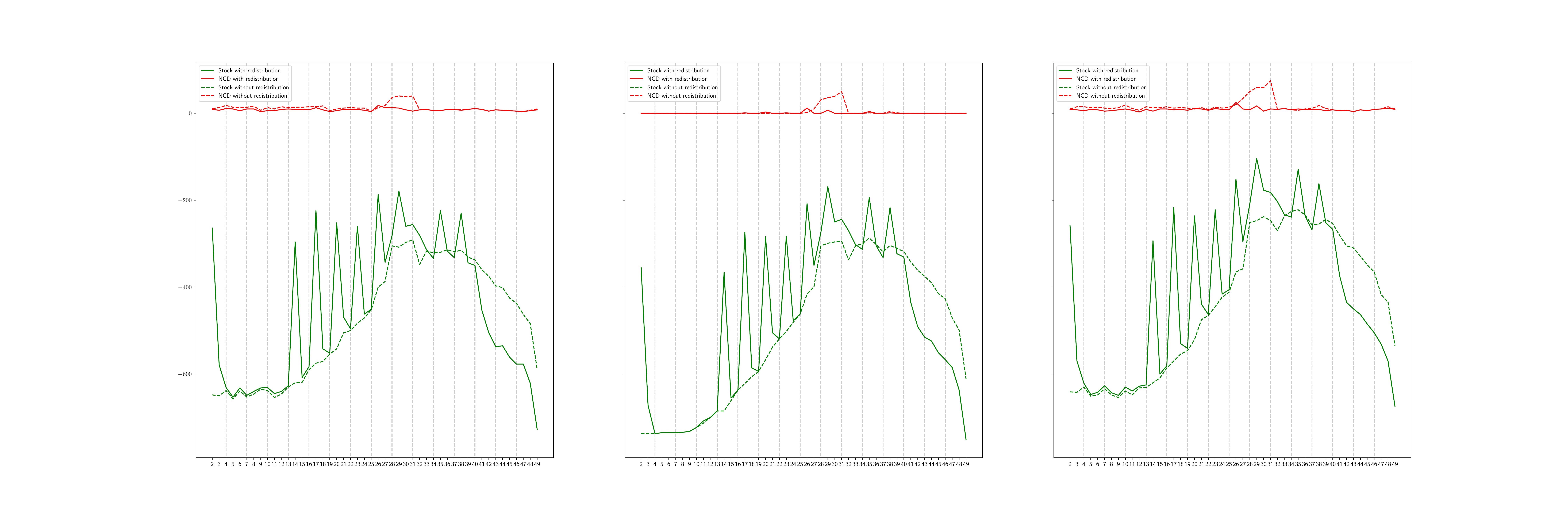}
\caption{Non-covered demand (red lines) and available stock (green lines) at each time period, if the \textit{optimistic} (left), \textit{real} (center) or \textit{pessimistic} (left) scenario occurs, with (continuous line) and without (dashed line) redistribution, in Andalucía, for C-graph and objective function $\Phi_4$.}\label{fig:ScenariosAnd}
\end{figure}

\subsubsection*{Complete Graph vs. Logistic Center Graph} 

We include in this section a comparison of the non-covered demand for the LC- and C-graphs. In Figure \ref{fig:CLvsCompMad}, we do it for the case of Madrid and objective functions $ \Phi_2^{\rm Regret}$ and $\Phi_4^{\rm Regret}$, and in Figure \ref{fig:CLvsCompAnd} for the case of Andalucía and objective functions $\Phi_2$ and $\Phi_4$. The red line represents the non-covered demand for the LC-graphs, and the blue line for the C-graphs. For the case of Madrid, the behavior under both graph structures is quite similar. Using logistic centers to redistribute the ventilators seems to perform a bit worse than not using them for more cases (according to the settings we used in the numerical study), but not for all of them. This could be due to the need of considering more than one logistic center since the demand is very high. For the case of Andalucía, using logistic centers seems to perform better and to lead to less high demand peaks. For example, when minimizing the total non-covered expected demand, that is, for objective function $\Phi_4$ (right graphic in Figure \ref{fig:CLvsCompAnd}), using logistic centers results much better than not using them.

\begin{figure}[h]

\hspace*{-0.75cm}\includegraphics[scale=0.3]{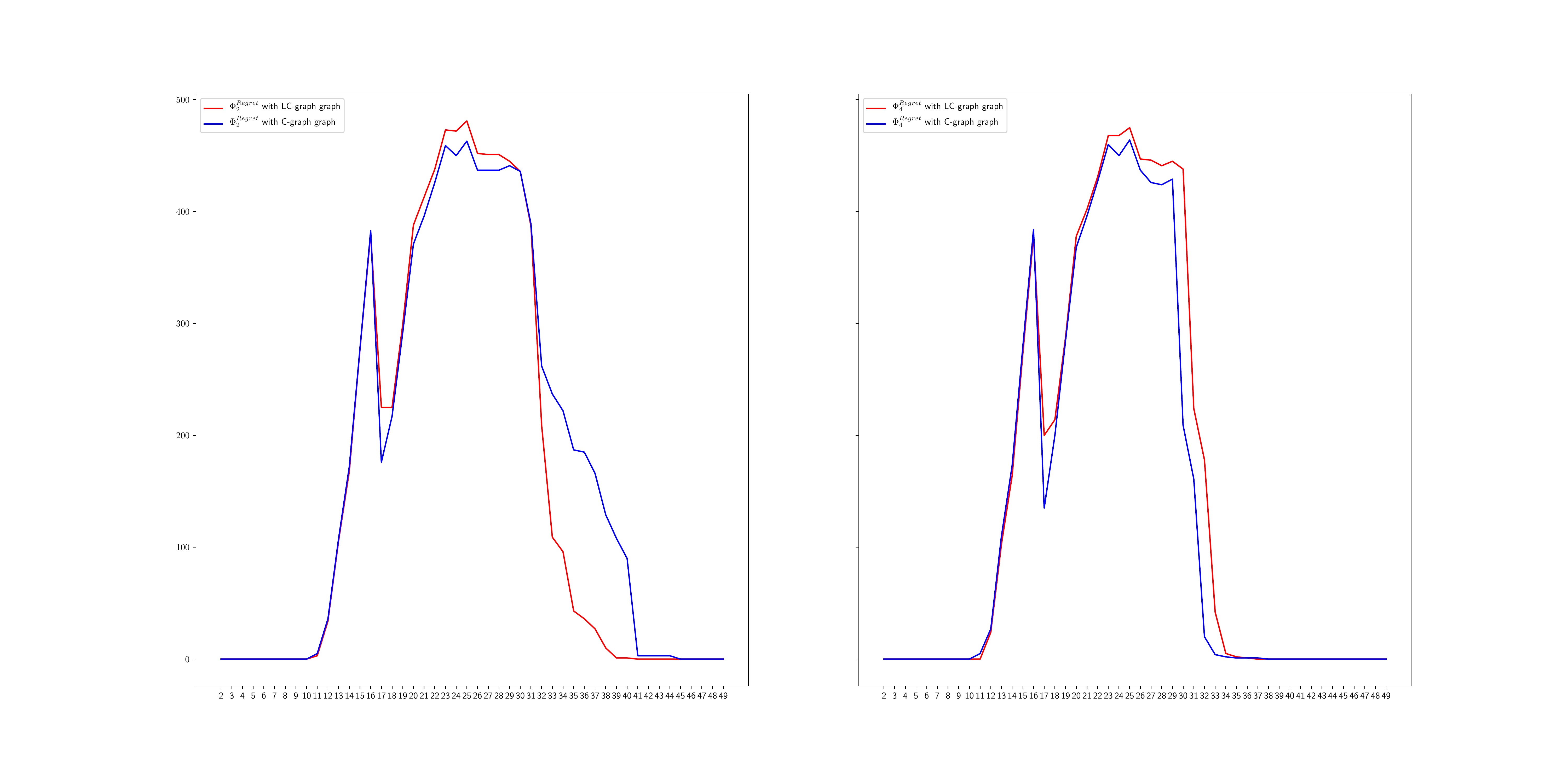}

\caption{Non-covered demand in Madrid at each time period for  the  LC-graph (red) and C-graph (blue), for objective functions $ \Phi_2^{\rm Regret}$ (left) and  $\Phi_4^{\rm Regret}$ (right) if the \textit{real} scenario occurs when considereing redistribution.}\label{fig:CLvsCompMad}
\end{figure}

\begin{figure}[h]
\hspace*{-0.75cm}\includegraphics[scale=0.3]{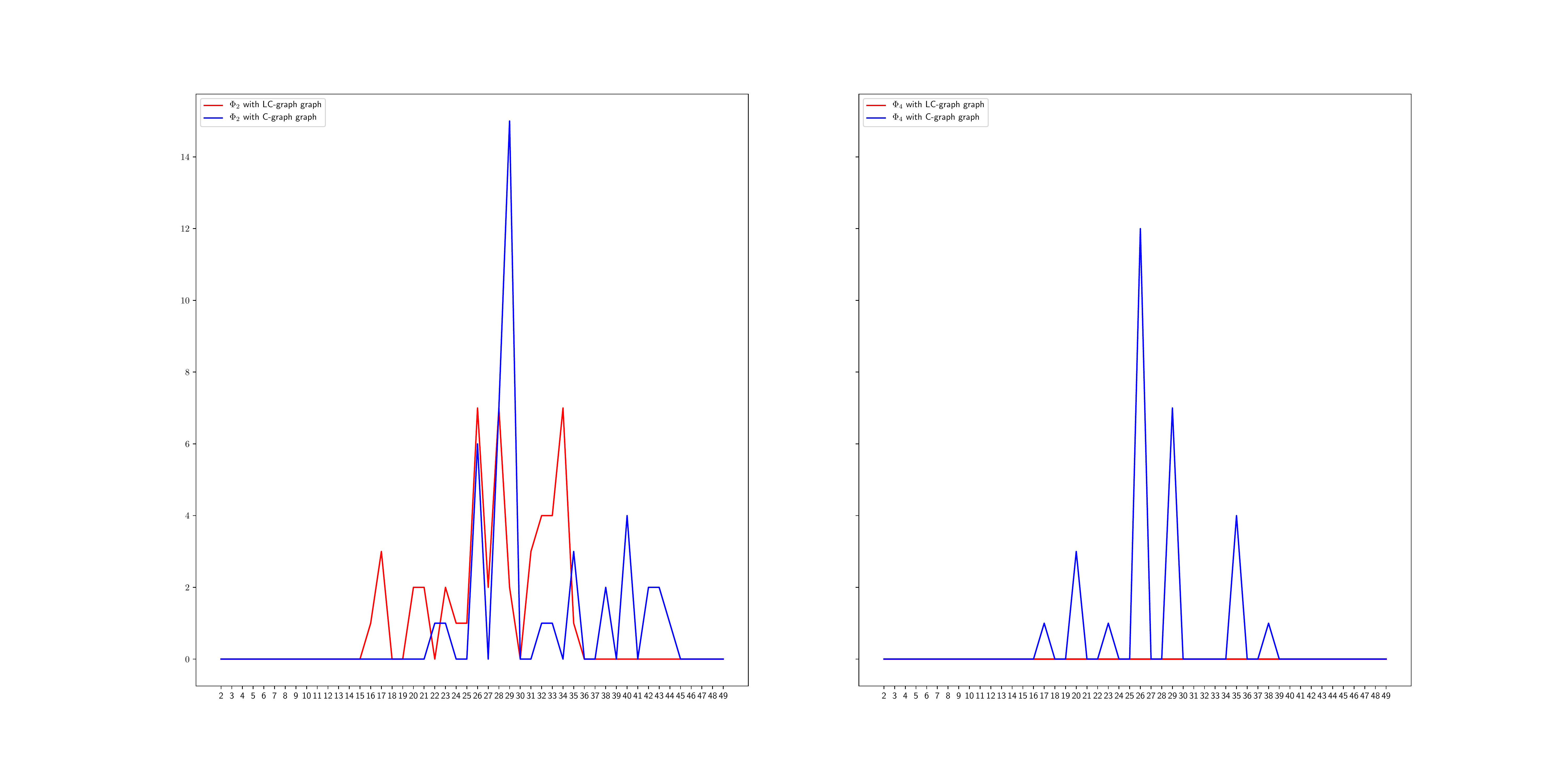}
\caption{Non-covered demand in Andalucía at each time period for  the  LC-graph (red) and C-graph (blue), for objective functions $ \Phi_2$ (left) and $\Phi_4$ (right), if the \textit{real} scenario occurs when considereing redistribution.}\label{fig:CLvsCompAnd}
\end{figure}

\subsubsection*{Different Criteria}

We compare in this last section the regret version (dashed line) versus the not regret version (continuous line) of different objective functions for Madrid in Figure \ref{fig:regvsMad}, and for Andalucía in Figure \ref{fig:regvsAnd}. We can observe that for the cases of Madrid included in Figure \ref{fig:regvsAnd} (objective function $\Phi_2$ and LC-graph (left) and $\Phi_4$ and C-graph (right)) the total non-covered demand of each objective function practically coincides for the regret and not regret version. The reason for this could be that the existence of few available stock make that the redistribution options are scarce and similar in both cases. For the case of Andalucía, the version without regret tends to perform better for most of the cases if real scenario occurs. However, there exist cases, see for example the case of $\Phi_2$ vs $\Phi_2^{\rm Regret}$ and C-graph (right graphic), for which there exist periods that the regret version covers more demand than the version without regret. 

In general, this tendency of the regret version to perform worse than the version without regret is maybe due to the election we made on the demand scenarios. Note that the real scenario, the one for which we are representing the NCD, is approximately an average of the other the other two scenarios, which benefits the objective functions without regret that average over the three scenarios. However, in cases in which the demand scenarios differ more, the regret version may perform better.

\begin{figure}[h]
\hspace*{-0.75cm}\includegraphics[scale=0.3]{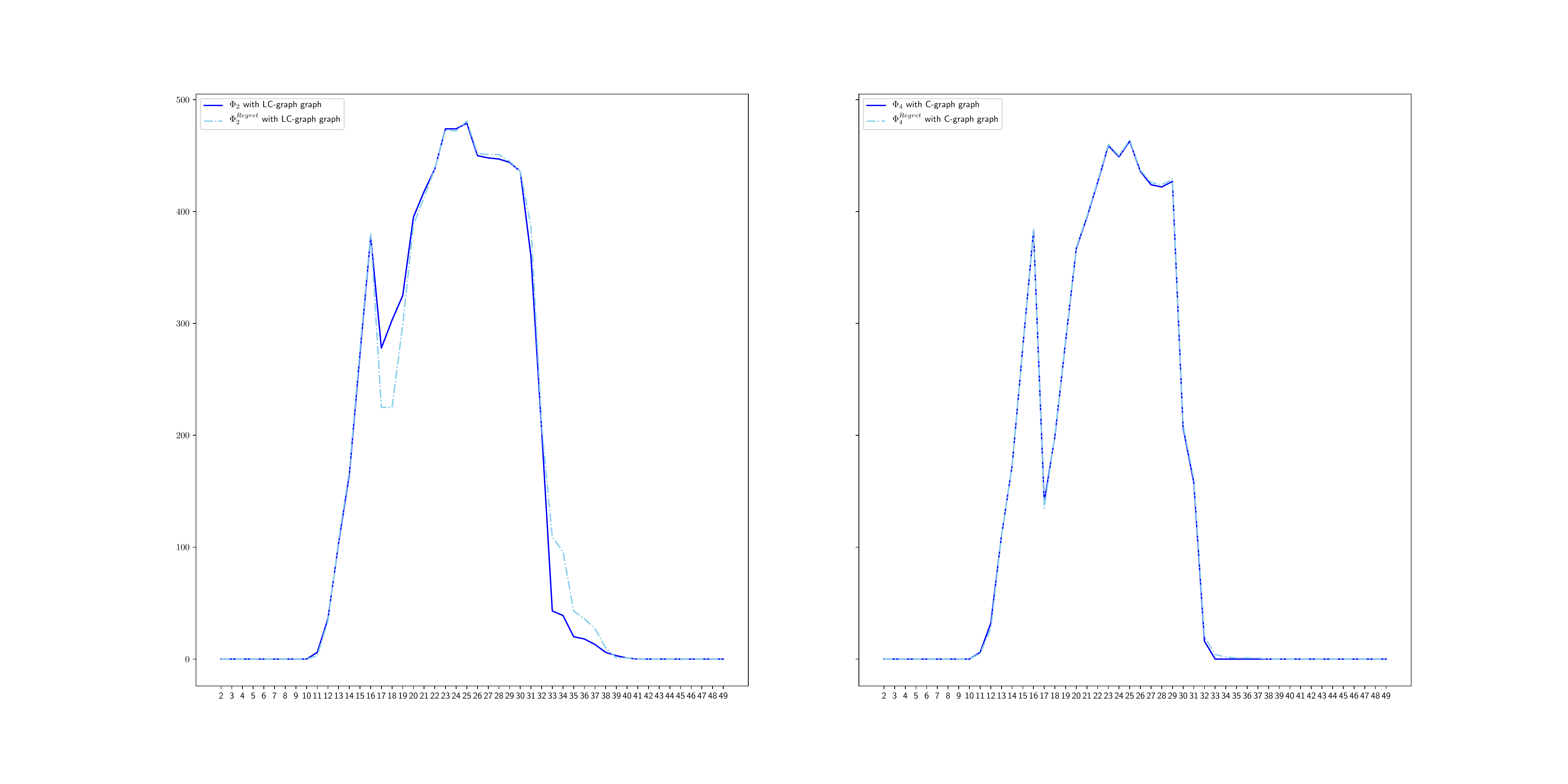}
\caption{Non-covered demand in Madrid, at each time period, for the regret (dashed line) and not regret (continuous line) versions of objective functions $\Phi_2$ (left) and $\Phi_4$ (right), for the LC-graph (left) and C-graph (right) if the \textit{real} scenario occurs when considereing redistribution.}\label{fig:regvsMad}
\end{figure}

\begin{figure}[h]
\hspace*{-0.75cm}\includegraphics[scale=0.3]{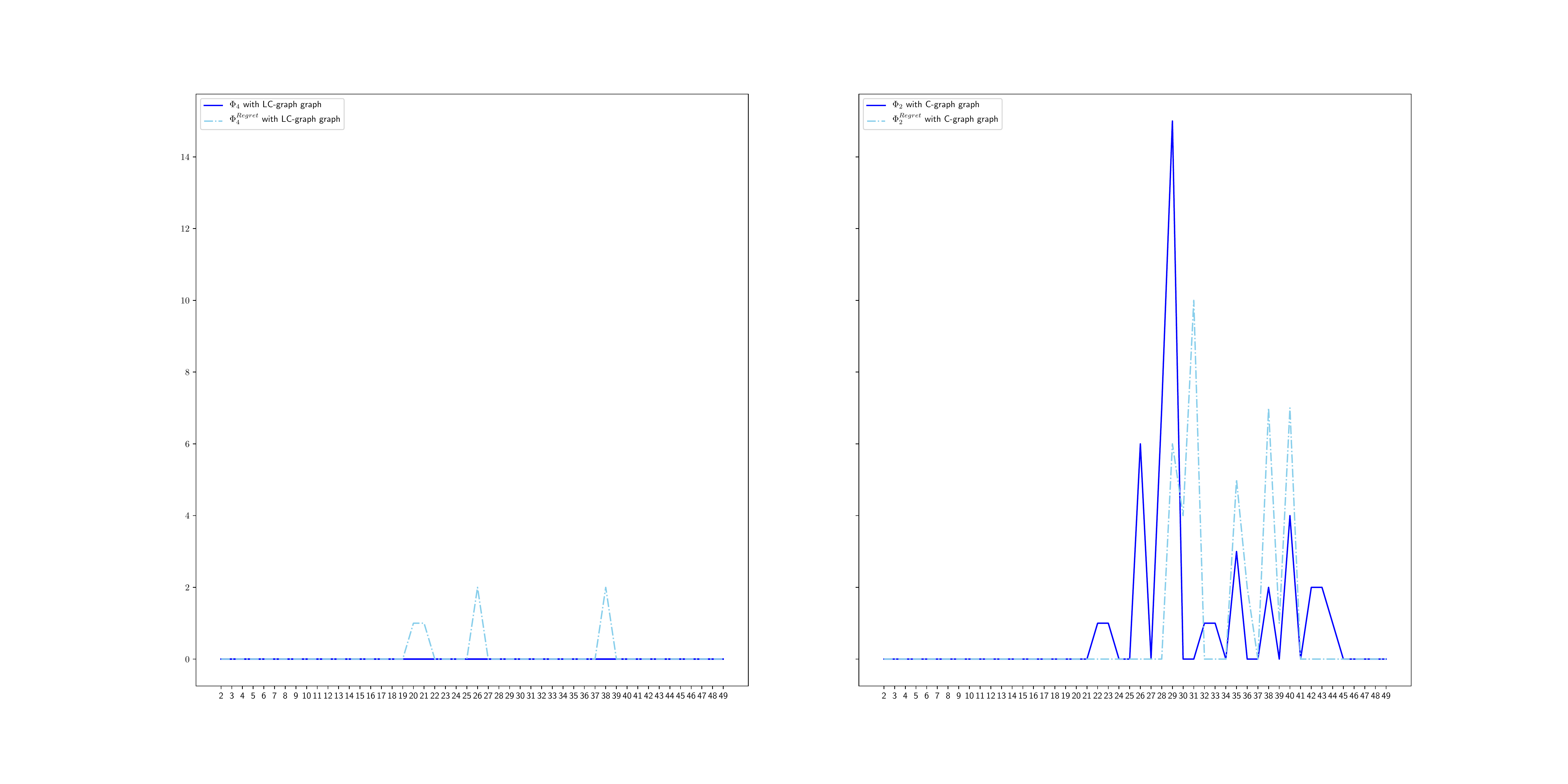}
\caption{Non-covered demand in Andalucía, at each time period, for the regret (dashed line) and not regret (continuous line) versions of objective functions $\Phi_4$ (left) and $\Phi_2$ (right), for the LC-graph (right) and C-graph (left) if the \textit{real} scenario occurs when considereing redistribution.}\label{fig:regvsAnd}
\end{figure}

\section{Conclusions and Extensions}\label{sec:6}

In this paper we propose a decision aid tools to determine optimal distribution and sharing strategies in emergency situations. The approach is motivated by the COVID-19 pandemic in which, in its first wave, a lack of emergency health equipments provoque hospitals collapse, and many patients were not adequately treated. Our approach allows decide how to distribute equipments through a given network of units along a time horizon and how to share the extra stock received at some of the periods, such that a global measure of the (stochastic) demand of the units is minimized. We provide an unified Mixed Integer Linear Programming formulation for the problem, able to model different distribution networks and different objective (robust and min-max regret) functions. We also propose a divide-et-conquer math-heuristic approach for the problem that provides feasible solutions of the problem in reasonable computational times. Finally, we analyze the case of the lack of mechanical invasive ventilators during the COVID-19 health crisis in two different Spanish regions. We run our approach with different settings, obtaining as the main conclusion that applying our approach leads to a significant increase in the number of severe patients that can be rightly assisted. Furthermore, we observed that an optimal redistribution of the extra stock must be based on the demand and not on the distribution of population.

An extension of our approach will be the topic of a forthcoming paper. In particular, we observe that the use of logistic centers when distributing equipments in emergency situations may be advisable in many cases, since they may allow more adequate distribution, may allow to store non-used equipments waiting for demand, quicker deliveries, etc. However, some regions do not still have the infrastructures of those centers or the ones that they have are not sufficient. Also, in emergency situations, it may be useful to use \textit{field logistic centers} during certain periods to improve the distribution of equipments during the demand peaks. In those cases, apart of deciding the amounts to be delivered and shared, one must decide where to locate new logistic centers. Our approach could be extended to this case with major modifications. In particular, the distribution network would not be known and is part of the decision (it depends on the position of the logistic centers), and then, it must be incorporated to the model, increasing considerably the complexity of the approach proposed in this paper. Observe that the model becomes a hub-and-spoke location problem in which the flows associated to commodities are decision variables of the model. Also, an interesting extension of this approach would be the incorporation of ordered weighted averaging  aggregations of the non-covered demands to construct solutions under other robust objective functions (see e.g., \cite{BPE13,BPE14} for an application of this type of aggregations in other logistic problems).

\section*{Acknowledgements}

The first and second autor were partially supported by research group SEJ-584 (Junta de Andaluc\'ia) and the third author by research group FQM-331 (Junta de Andaluc\'ia). The second author was supported by Spanish Ministry of Education and Science grant number PEJ2018-002962-A. The first and third authors were partially supported by project  MTM2016-74983-C2-1-R (MINECO, Spain) and projects P18-FR-1422 and US-1256951 (Junta de Andaluc\'ia). First author was also supported by Proyect \textit{NetMeetData} (Fundaci\'on BBVA - Big Data).

\renewcommand\refname{Sources}

\end{document}